\documentclass[11pt,reqno]{amsart}

\usepackage[latin1]{inputenc}
\usepackage{amsmath,amsthm,amssymb,amscd,mathrsfs,mathtools,color,esint}
\usepackage[perpage]{footmisc}
\definecolor{darkgreen}{rgb}{0.0, 0.6, 0.13}

\usepackage{hyperref}

\newtheorem{thm}{Theorem}[section]
 \newtheorem{cor}[thm]{Corollary}
 \newtheorem{lem}[thm]{Lemma}
 \newtheorem{prop}[thm]{Proposition}
\newtheorem{claim}[thm]{Claim}
 \theoremstyle{definition}
 \newtheorem{df}[thm]{Definition}
 \theoremstyle{remark}
 \newtheorem{rem}[thm]{Remark}
 
 \numberwithin{equation}{section}

\newcommand{\defeq}{\mathrel{\mathop:}=}

\newcommand{\R}{\mathbb R}
\newcommand{\T}{\mathbb T}
\newcommand{\Z}{\mathbb Z}
\newcommand{\Jm}{\mathcal J}
\newcommand{\Rm}{\mathcal R}
\newcommand{\Lm}{\mathcal L}
\newcommand{\Tm}{\mathcal{T}}
\newcommand{\Qm}{\mathcal Q}
\newcommand {\Nm}{\mathcal{N}}
\newcommand{\Cm}{\mathcal C}

 \newcommand{\Addresses}{{
  \bigskip
  \footnotesize

  Yu Deng, \textsc{Department of Mathematics, University of Southern California, Los Angeles, CA, USA}\par\nopagebreak
  \textit{E-mail address}: \texttt{yudeng@usc.edu}

  \bigskip

  Zaher Hani, \textsc{Department of Mathematics, University of Michigan, Ann Arbor, MI, USA}\par\nopagebreak
 \textit{E-mail address}: \texttt{zhani@umich.edu}
  
  }}

\begin{document}
  \author{Yu Deng and Zaher Hani}
 \title{On the derivation of the wave kinetic equation for NLS}
 \begin{abstract}
A fundamental question in wave turbulence theory is to understand how the ``wave kinetic equation" (WKE) describes the long-time dynamics of its associated nonlinear dispersive equation. Formal derivations in the physics literature, dating back to the work of Pieirls in 1928, suggest that such a kinetic description should hold (for well-prepared random data) at a large \emph{kinetic time scale} $T_{\mathrm{kin}} \gg1$, and in a limiting regime where the size $L$ of the domain goes to infinity, and the strength $\alpha$ of the nonlinearity goes to 0 (weak nonlinearity). For the cubic nonlinear Schr\"odinger equation, $T_{\mathrm{kin}}=O(\alpha^{-2})$ and $\alpha$ is related to the conserved mass $\lambda$ of the solution via $\alpha=\lambda^2 L^{-d}$. 
\medskip
 
 In this paper, we study the rigorous justification of this monumental statement, and show that the answer seems to depend on the particular ``\emph{scaling law}" in which the $(\alpha, L)$ limit is taken, in a spirit similar to how the Boltzmann-Grad scaling law is imposed in the derivation of Boltzmann's equation. In particular, there appears to be \emph{two} favorable scaling laws: when $\alpha$ approaches $0$ like $L^{-\varepsilon+}$ or like $L^{-1-\frac{\varepsilon}{2}+}$ (for arbitrary small $\varepsilon$), we exhibit the wave kinetic equation up to timescales $O(T_{\mathrm{kin}}L^{-\varepsilon})$, by showing that the relevant Feynman diagram expansions converge absolutely (as a sum over paired trees). For the other scaling laws, we justify the onset of the kinetic description at timescales $T_*\ll T_{\mathrm{kin}}$, and identify specific interactions that become very large for times beyond $T_*$. In particular, the relevant tree expansion diverges absolutely there. In light of those interactions, extending the kinetic description beyond $T_*$ towards $T_{\mathrm{kin}}$ for such scaling laws seems to require new methods and ideas.
\end{abstract}

 \maketitle
 
 \tableofcontents
 \section{Introduction}
 
 The \emph{kinetic framework} is a general paradigm that aims to extend Boltzmann's kinetic theory for dilute gases to other types of microscopic interacting systems. This approach has been highly informative, and became a corner stone of the theory of nonequilibrium statistical mechanics for a large body of systems \cite{Spohn1,SpohnBook}. In the context of nonlinear dispersive waves, this framework was initiated in the first half of the past century \cite{Peierls} and developed into what is now called \emph{wave turbulence theory} \cite{ZLFBook, Nazarenko}. There, waves of different frequencies interact nonlinearly at the microscopic level, and the goal is to extract an effective macroscopic picture of how the energy densities of the system evolve. 
 
 The description of such an effective evolution comes via the \emph{wave kinetic equation} (WKE), which is the analog of Boltzmann's equation for nonlinear wave systems \cite{SpohnBE}. Such kinetic equations have been derived at a formal level for many systems of physical interest (NLS, NLW, water waves, plasma models, lattice crystal dynamics, etc.~cf.~\cite{Nazarenko} for a textbook treatment) and are used extensively in applications (thermal conductivity in crystals \cite{SpohnPBE}, ocean forecasting \cite{Janssen, WMO}, etc.). This kinetic description is conjectured to appear in the limit where the number of (locally interacting) waves goes to infinity, and an appropriate measure of the interaction strength goes to zero (weak nonlinearity\footnote{It is for this reason that the theory is sometimes called \emph{weak turbulence theory.}}). In such kinetic limits, the total energy of the whole system often diverges. 
 
 The fundamental mathematical question here, which also has direct consequences for the physical theory, is to provide a rigorous justification of such wave kinetic equations starting from the microscopic dynamics given by the nonlinear dispersive model at hand. The importance of such an endeavor stems from the fact that it allows to understand the exact regimes and the limitations of the kinetic theory, which has long been a matter of scientific interest (see \cite{DLN, Mordant}). A few mathematical investigations have recently been devoted to study problems of this spirit \cite{Faou, BGHS2, LukSpohn}, that yielded some partial results and useful insights.
 
 \medskip
 
This manuscript continues the investigation initiated in \cite{BGHS2} aimed at providing a rigorous justification of the wave kinetic equation corresponding to the nonlinear Schr\"odinger equation,
\[
i \partial_t v - \Delta v + |v|^{2} v=0.
\]
As we shall explain later, the sign of the nonlinearity has no effect on the kinetic description, so we choose the defocusing sign for concreteness. The natural setup for the problem is to start with a spatial domain given by a torus $\T^d_L$ of size $L$, which approaches infinity in the thermodynamic limit we seek. This torus can be rational or irrational, which amounts to rescaling the Laplacian into
\[
\Delta_\beta \defeq  \frac{1}{2\pi}  \sum\limits_{i=1}^d \beta_i \partial_i^2, \qquad \beta \defeq  (\beta_1,\dots,\beta_d)\in [1,2]^d,
\]
and taking the spatial domain to be the standard torus of size $L$, namely $\T^d_L=[0,L]^d$ with periodic boundary conditions. With this normalization, an irrational torus would correspond to taking the $\beta_j$ to be rationally independent. Our results cover both cases, and in part of them $\beta$ is assumed to be generic, i.e. avoiding a set of Lebesgue measure 0.

The strength of the nonlinearity is related to the characteristic size $\lambda$ of the initial data (say in the conserved $L^2$ space). Adopting the ansatz $v=\lambda u$, we arrive at the following equation:
 \[
 \tag{NLS} \label{NLS}
\begin{cases}
i \partial_t u - \Delta_\beta u + \lambda^{2} |u|^{2} u=0,  \quad x\in \mathbb{T}_L^d = [0,L]^d,   \\[.6em]
u(0,x) = u_{\textrm{in}}(x).
\end{cases}
\]

\medskip

The kinetic description of the longtime behavior is akin to a law of large numbers, and as such one has to start with a random distribution of the initial data. Heuristically, a randomly distributed, $L^{2}$-normalized field would (with high probability) have a roughly uniform spatial distribution, and consequently an $L_x^\infty$ norm $\sim L^{-d/2}$. This makes the strength of the nonlinearity in \eqref{NLS} comparable to $\lambda^2 L^{-d}$ (at least initially\footnote{Formal derivations of the wave kinetic equation often involve heuristic arguments (like a propagation of quasi-gaussianity of the initial data through time), which effectively imply that the strength of the nonlinearity stays $\sim \lambda^2 L^{-d}$. Such heuristic arguments are hard to justify rigorously, however this bound on the nonlinearity strength will be propagated and proved as a consequence of our estimates. }), which motivates us to introduce the quantity $$\alpha=\lambda^2L^{-d},$$ and phrase the results in terms of $\alpha$ instead of $\lambda$. 
The kinetic \emph{conjecture} states that at sufficiently long timescales, the effective dynamics of the Fourier-space mass density $\mathbb E |\widehat u(t, k)|^2$ ($k \in \Z^d_L=L^{-1}\Z^d$) is well approximated -in the limit of large $L$ and vanishing $\alpha$- by (an appropriately scaled) solution $n(t, \xi)$ of the following wave kinetic equation (WKE):
\[
\tag{WKE}\label{WKE}
\begin{split}
\partial_t n(t, \xi) =&\mathcal K\left(n(t, \cdot)\right), \\
\mathcal K(\phi)(\xi):=& \int_{\substack{(\xi_1, \xi_2, \xi_3)\in \R^{3d}\\\xi_1-\xi_2+\xi_3=\xi}} \phi \phi_1 \phi_2 \phi_3\left(\frac{1}{\phi_1}-\frac{1}{\phi_2}+\frac{1}{\phi_3}-\frac{1}{\phi}\right)\delta_{\R}(|\xi_1|_\beta^2-|\xi_2|_\beta^2+|\xi_3|_\beta^2-|\xi|_\beta^2)\, d\xi_1 d\xi_2 d\xi_3,
\end{split}
\]
where we used the shorthand notations $\phi_j:=\phi(\xi_j)$ and $|\xi|^2_\beta=\sum_{j=1}^d \beta_j (\xi^{(j)})^2$ for $\xi=(\xi^{(1)},\cdots,\xi^{(d)})$. More precisely, one expects this approximation to hold at the kinetic timescale $T_{\mathrm{kin}}\sim \alpha^{-2}=\frac{L^{2d}}{\lambda^4}$, in the sense that 
\begin{equation}\label{approximation}
\mathbb E |\widehat u(t, k)|^2 \approx n\left(\frac{t}{T_{\mathrm{kin}}}, k\right), \qquad \textrm{as } L\to \infty, \alpha \to 0.
\end{equation}

Of course, for such an approximation to hold at time $t=0$, one has to start with a \emph{well-prepared} initial distribution for $\widehat u_{\textrm{in}}(k)$ as follows: Denoting by $n_{\textrm{in}}$ the initial data for \eqref{WKE}, we assume
\begin{equation}\label{wellprepared}
\widehat u_{\mathrm{in}}(k)=\sqrt{n_{\textrm{in}}(k)} \, \eta_{k}(\omega)
\end{equation}
where $\eta_{k}(\omega)$ are mean-zero complex random variables satisfying $\mathbb E |\eta_k|^2=1$. In what follows, $\eta_k(\omega)$ will be independent, identically distributed complex random variables, such that the law of each $\eta_k$ is either the normalized complex Gaussian, or the uniform distribution on the unit circle $|z|=1$.

\medskip 

Before stating our results, it is worth remarking on the regime of data and solutions covered by this kinetic picture in comparison to previously studied and well-understood regimes in the nonlinear dispersive literature. For this, let us look back at the (pre-ansatz) NLS solution $v$, whose conserved energy is given by
$$
\mathcal E[v]:=\int_{\T^d_L} \frac{1}{2}|\nabla v|^2 +\frac{1}{4}|v|^4 \; dx.
$$
We are dealing with solutions having $L^\infty$ norm of $O(\sqrt \alpha)$ (with high probability), and whose total mass is $O(\alpha L^d)$, in a regime where $\alpha$ is vanishingly small and $L$ is asymptotically large. These bounds on the solutions are true initially as we explained above, and will be propagated in our proof. In particular, the mass and energy are very large and will diverge in this kinetic limit, as is common when taking thermodynamic limits \cite{Ruelle, Minlos}. 
Moreover, the potential part of the energy is dominated by the kinetic part (the former being of size $O(\alpha^3 L^d)$ whereas the latter of size $O(\alpha L^d)$), which explains why there is no distinction between the defocusing and focusing nonlinearity in the kinetic limit. It would be interesting to see how the kinetic framework can be extended to regimes of solutions which are sensitive to the sign of the nonlinearity; this has been investigated in the physics literature (see for example \cite{DNPZ, FGMW, ZKPR}).

\medskip

\subsection{Statement of the results} 
It is not a priori clear how the limits $L\to \infty$ and $\alpha \to 0$ need to be taken for \eqref{approximation} to hold, and if there's an additional scaling law (between $\alpha$ and $L$) that needs to be satisfied in the limit. In comparison, such scaling laws are imposed in the rigorous derivation of Boltzmann's equation \cite{Lanford, CIP, GSRT}, which is derived in the so-called Boltzmann-Grad limit \cite{Grad}, namely the number $N$ of particles goes to $\infty$, while their radius $r$ goes to $0$ in such a way that $Nr^{d-1}\sim O(1)$. To the best of our knowledge, this central point has not been adequately addressed in the wave turbulence literature. 

Our results seem to suggest some key differences depending on the chosen scaling law. Roughly speaking, we identify \emph{two special scaling laws} for which we are able to justify the approximation \eqref{approximation} up to timescales $L^{-\varepsilon} T_{\textrm{kin}}$ for any arbitrarily small $\varepsilon>0$. For other scaling laws, we identify significant absolute divergences in the power series expansion for $\mathbb E |\widehat u(t, k)|^2$ at much earlier times. As such, we can only justify this approximation at such shorter times (which are still better than those in \cite{BGHS2}). In these cases, whether or not \eqref{approximation} holds up to timescales $L^{-\varepsilon} T_{\textrm{kin}}$ depends on whether such series converge conditionally instead of absolutely, and as such would require new methods and ideas, as we explain below.

\medskip
We start with our first theorem identifying the two favorable scaling laws.

\begin{thm}\label{thm1}
Let $d\geq 2$, and $\beta\in [1,2]^d$ be arbitrary. Suppose that $n_{\mathrm{in}} \in \mathcal S(\R^d \to [0, \infty))$ is Schwartz\footnote{In fact, only a finite amount of decay and smoothness is needed on $n_{\mathrm{in}}$. We chose $n_{\mathrm{in}}\in \mathcal S$ to simplify the exposition and avoid minor distracting technicalities.}, and $\eta_{k}(\omega)$ are independent, identically distributed complex random variables, such that the law of each $\eta_k$ is either complex Gaussian with mean $0$ and variance $1$, or the uniform distribution on the unit circle $|z|=1$. Assume well-prepared initial data $u_{\mathrm{in}}$ for \eqref{NLS} as in \eqref{wellprepared}.

\medskip 

Fix $0<\varepsilon<1$ (in most interesting cases $\varepsilon$ will be small), and set $\alpha=\lambda^2L^{-d}$ to be the characteristic strength of the nonlinearity. If $\alpha$ has the scaling law $\alpha\sim L^{-\varepsilon+}$ or $\alpha\sim L^{-1-\frac{\varepsilon}{2}+}$ (where $c+$ represents a number strictly bigger than and sufficiently close to the given $c$), then there holds
\begin{equation}\label{approx2}
\mathbb E |\widehat u(t, k)|^2 =n_{\mathrm{in}}(k)+\frac{t}{T_{\mathrm{kin}}}\mathcal K(n_{\mathrm{in}})(k)+o_{\ell^\infty_k}\left(\frac{t}{T_{\mathrm {kin}}}\right)_{L \to \infty}
\end{equation}
for all $L^{0+} \leq t \leq L^{-\varepsilon} T_{\mathrm{kin}}$, where $T_{\mathrm{kin}}=\alpha^{-2}$, $\mathcal K$ is defined in \eqref{WKE}, and $o_{\ell^\infty_k}\left(\frac{t}{T_{\mathrm {kin}}}\right)_{L \to \infty}$ is a quantity that is bounded in $\ell^\infty_k$ by $L^{-\theta} \frac{t}{T_{\mathrm {kin}}}$ for some $\theta>0$.

\end{thm}
 
We remark that, in the time interval of approximation above, the right hand sides of \eqref{approximation} and \eqref{approx2} are equivalent. Also note that any type of scaling law of the form $\alpha\sim L^{-s}$, gives an upper bound of $t\leq L^{-\varepsilon}T_{\mathrm{kin}}\sim L^{2s-\varepsilon}$ for the times considered. Consequently, for the two scaling laws in the above theorem, the time $t$ always satisfies $t\ll L^{2}$, and it is for this reason that the rationality-type of the torus is not relevant. As will be clear below, no similar results can hold for $t\gg L^2$ in the case of a rational torus\footnote{Even at the endpoint case where $t\sim L^2$, the number theoretic components of the proof would yield different answers than those anticipated by the theory. See Section \ref{number theory section}.}, as this would require rational quadratic forms to be equidistributed on scales $\ll1$, which is impossible. However, if the aspect ratios $\beta$ are assumed to be generically irrational, then one can access equidistribution scales that are as small as $L^{-d+1}$ for the resulting irrational quadratic forms \cite{Bourgain, BGHS2}. This allows us to consider scaling laws for which $T_{\mathrm{kin}}$ can be as big as $L^{d-}$ on generically irrational tori.
  
  \medskip

  \begin{rem}
The $\varepsilon\to 0$ limit of the first scaling law $\alpha\sim L^{-\varepsilon}$ effectively corresponds to taking the large box limit $L\to \infty$ before taking the weak nonlinearity limit $\alpha\to 0$. This is the usual order taken in formal derivations (see \cite{Nazarenko} for instance). 
 \end{rem}

The following theorem covers general scaling laws, including the ones that can only be accessed for the generically irrational torus. By a simple calculation of exponents, we can see that it implies Theorem \ref{thm1}:
\begin{thm}\label{thm2}
With the same assumptions as in Theorem \ref{thm1}, we impose the following conditions on $(\alpha, L, T)$ for some $\delta>0$: 
\begin{equation}\label{adm1}T\leq
\left\{
\begin{split}&L^{2-\delta},&&\mathrm{\ if\ }\beta_i\mathrm{\ arbitrary},\\
&L^{d-\delta},&&\mathrm{\ if\ }\beta_i\mathrm{\ generic},
\end{split}\right.\qquad\alpha\leq \left\{
\begin{split}&L^{-\delta}T^{-1},&\mathrm{\ if\ }T&\leq L,\\
&L^{-1-\delta},&\mathrm{\ if\ }L&\leq T\leq L^2,\\
&L^{1-\delta}T^{-1},&\mathrm{\ if\ } T&\geq L^2.
\end{split}
\right.
\end{equation}
Then \eqref{approx2} holds for all $L^{\delta} \leq t \leq T$.
\end{thm}

  \begin{figure}[h!]
  \includegraphics[width=.58\linewidth]{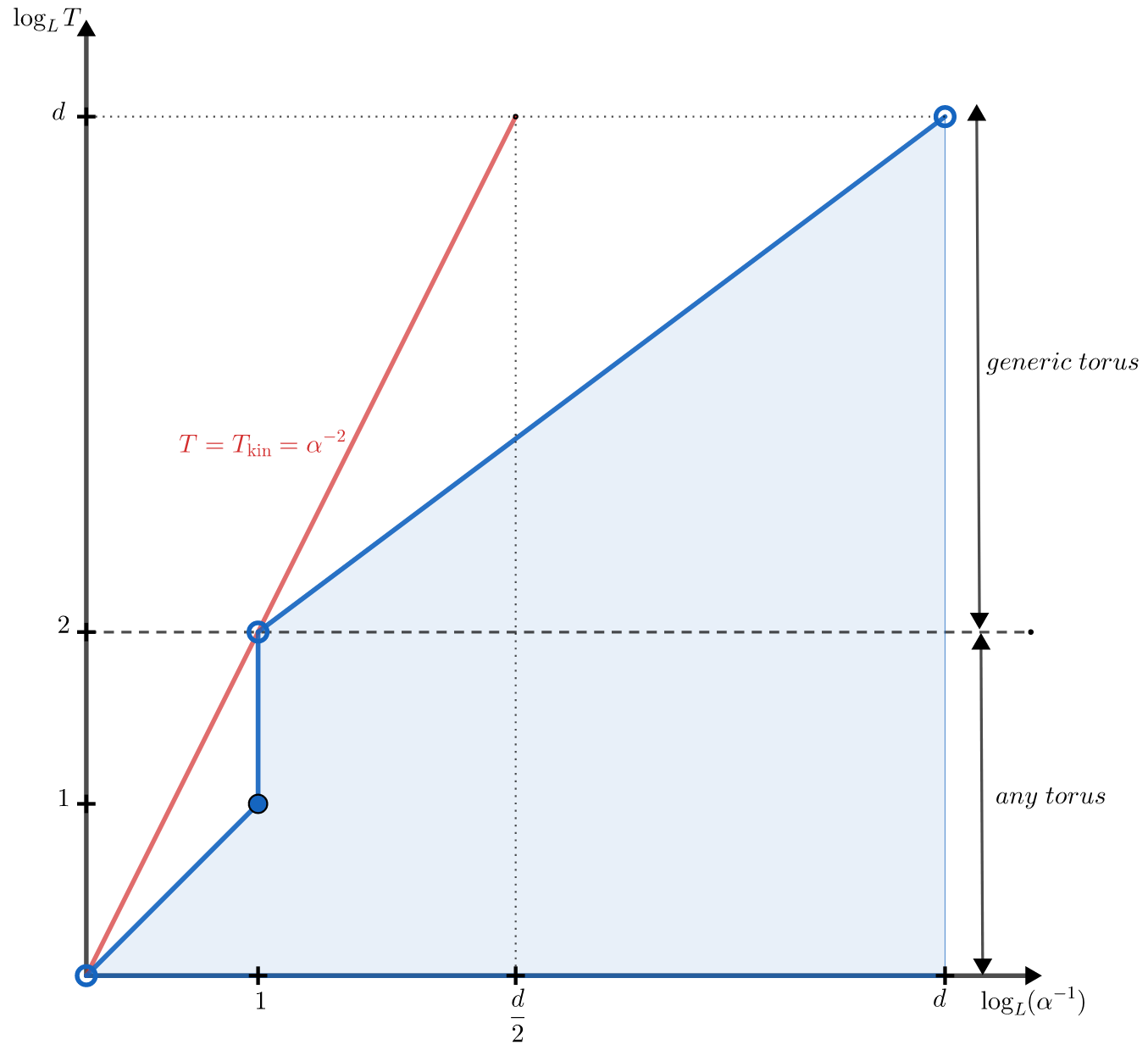}
  \caption{Admissible range for $(\alpha, L, T)$ in the $(\log_L (\alpha^{-1}),\log_L T)$ plot when $d\geq 3$. The blue region is the range of our theorem (up to $\varepsilon$ endpoint accuracy). The red line denotes the case when $T=T_{\mathrm{kin}}=\alpha^{-2}$, which our blue region touches at two points corresponding to $T\sim 1$ and $T\sim L^{2}$.}
  \label{fig:admissible}
\end{figure}

It is best to read this theorem in terms of $\left(\log_L (\alpha^{-1}),\log_L T\right)$ plot in Figure \ref{fig:admissible}. The kinetic conjecture corresponds to justifying the approximation in \eqref{approximation} up to timescales $T\lesssim T_{\mathrm{kin}}=\alpha^{-2}$. As we shall explain below, timescale $T\sim T_{\mathrm{kin}}$ represents a \emph{critical scale} for the problem from a probabilistic point of view. This is depicted by the red line of this figure, and the region below this line corresponds to a \emph{probabilistically-subcritical} regime (see Section \ref{probsc}). The shaded blue region corresponds to the $(\alpha, T)$ region in Theorem \ref{thm2} neglecting $\delta$ losses. This region touches the line $T=\alpha^{-2}$ at the two points corresponding to $(\alpha^{-1}, T)=(1, 1)$ and $(L, L^2)$, while the two scaling laws of Theorem \ref{thm1}, where $(\alpha^{-1},T)\sim(L^{\varepsilon-},L^{\varepsilon-})$ and $(\alpha^{-1},T)\sim (L^{1+\frac{\varepsilon}{2}-},L^{2-})$, approach these two points when $\varepsilon$ is small.

The above results rely on a diagrammatic expansion of the NLS solution in Feynman diagrams akin to a Taylor expansion. The shaded blue region depicting the result of Theorem \ref{thm2} corresponds to the cases when such a diagrammatic expansion is absolutely convergent for very large $L$. In the complementary region between the blue region and the line $T=T_{\textrm{kin}}$, we show that some (arbitrarily high-degree) terms of this expansion do not converge to 0 as their degree goes to $\infty$, which means that the diagrammatic expansion cannot converge absolutely in this region. As such, the only way for the kinetic conjecture to be true in the scaling regimes not included in Theorem \ref{thm1} is for those terms to exhibit a highly nontrivial cancellation which would make the series converge conditionally but not absolutely. 

Finally, we remark on the restriction in \eqref{adm1}. The upper bounds on $T$ on the left are necessary from number theoretic considerations: Indeed, if $T\gg L^2$ for a rational torus, or if $T\gg L^d$ for an irrational one, the exact resonances of the NLS equation dominate the quasi-resonant interactions that lead to the kinetic wave equation. As such, one should not expect the kinetic description to hold in those ranges of $T$ (see Lemma \ref{counting} and Section \ref{proof of theorem}). The second set of restrictions in \eqref{adm1} correspond exactly to the requirement that size of the Feynman diagrams of degree $n$ can be bounded by $\rho^n$ with some $\rho\ll 1$. 

\begin{rem}[Admissible scaling laws]\label{scaling rem}
The above restrictions on $T$ impose the limits of the admissible scaling laws, in which $\alpha\to 0$ and $L \to \infty$, for which the kinetic description of the longtime dynamics can appear. Indeed, since $T_{\mathrm{kin}}=\alpha^{-2}$, then the \emph{necessary} (up to $L^\delta$ factors) restrictions $T\ll L^{2-\delta}$ (resp. $T\ll L^{d-\delta}$) on the rational (resp. irrational) torus mentioned above, imply that one should only expect the above kinetic description in the regime where $\alpha \gtrsim L^{-1}$ (resp. $\gtrsim L^{-d/2}$). In other words, the kinetic description requires the nonlinearity to be weak, but not too weak! In the complementary regime of very weak nonlinearity, the exact resonances of the equation dominate the quasi-resonances, a regime referred to as discrete wave turbulence (see \cite{LN,Kart, Nazarenko}), in which different effective equations, like the (CR) equation in \cite{FGH, BGHS1}, can arize. 
\end{rem} 

\subsection{Ideas of the proof} As Theorem \ref{thm1} is a consequence of Theorem \ref{thm2}, we will focus on Theorem \ref{thm2} below. The proof of Theorem \ref{thm2} contains three components: (1) A long-time well-posedness result, where we expand the solution to (\ref{NLS}) into Feynman diagrams for sufficiently long time, up to a well-controlled error term; (2) Computation of $\mathbb E|\widehat u_k(t)|^2$ ($k \in \Z^d_L$) using this expansion, where we identify the leading terms and control the remainders; and (3) A number theoretic result that justifies the large box approximation, where we pass from the sums appearing in the expansion in component (2) to the integral appearing on the right-hand side of \eqref{WKE}.
 
The main novelty of this work is in the first component, which is the hardest. The second component follows similar lines to those in \cite{BGHS2}. Regarding the third component, the main novelty of this work is to complement the number theoretic results in \cite{BGHS2} (which only dealt with the generically irrational torus) by the cases of general tori (in the admissible range of time $T\ll L^2$). This provides an essentially full (up to $L^\varepsilon$ losses) understanding of the number theoretic issues arising in wave turbulence derivations for (NLS). As such, we will limit this introductory discussion to the first component. 
\subsubsection{The scheme and probabilistic criticality}\label{probsc} Though technically involved, the basic idea of the long-time well-posedness argument is in fact quite simple. Starting from (\ref{NLS}) with initial data (\ref{wellprepared}), we write the solution as
\begin{equation}
\label{expansion1}u=u^{(0)}+\cdots+u^{(N)}+\mathcal R_{N+1},
\end{equation} where $u^{(0)}=e^{-it\Delta_\beta}u_{\mathrm{in}}$ is the linear evolution, $u^{(n)}$ are iterated self-interactions of the linear solution $u^{(0)}$ that appear in a formal expansion of $u$, and $\mathcal R_{N+1}$ is a sufficiently regular remainder term.

Since $u^{(0)}$ is a linear combination of independent random variables, and each $u^{(n)}$ is a multilinear combination, each of them will behave strictly better (both linearly and nonlinearly) than its deterministic analogue (i.e. with all $\eta_k=1$). This is due to the well-known large deviation estimates, which yield a `square root' gain coming from randomness, akin to the Central Limit Theorem (for instance $\|u_{\mathrm{in}}\|_{L^\infty}$ is bounded by $L^{-d/2}\cdot\|u_{\mathrm{in}}\|_{L^2}$ in the probabilistic setting, as opposed to $1\cdot \|u_{\mathrm{in}}\|_{L^2}$ deterministically by Sobolev embedding, assuming compact Fourier support). This gain leads to a new notion of criticality for the problem, which can be defined\footnote{One can interpret the usual scaling criticality for \eqref{NLS} in the same way: It corresponds to the minimum regularity $s$ for which the first iterate of an $H^s-$normalized rescaled bump function like $N^{-s+\frac{d}{2}}\phi(Nx)$ is better bounded than the linear solution (comparing $|u|^pu$ to $\Delta u$ for such data gives the critical regularity $s_{\mathrm{critical}}=\frac{d}{2}-\frac{2}{p}$).} as the edge of the regime of $(\alpha, T)$ for which the iterate $u^{(1)}$ is better bounded than the iterate $u^{(0)}$. It is not hard to see that, $u^{(1)}$ can have size up to $O(\alpha\sqrt{ T})$ (in appropriate norms) compared to the $O(1)$ size of $u^{(0)}$ (See for example \eqref{lwp0-1} for $n=1$). This justifies the notion that $T\sim T_{\mathrm{kin}}=\alpha^{-2}$ corresponds to the \emph{probabilisitically-critical scaling}, whereas the timescales $T\ll T_{\mathrm{kin}}$ are \emph{subcritical}\footnote{It may be supercritical under the \emph{deterministic scaling}. See \cite{DNY2}, Section 1.3 for a discussion of these notions in the more customary context of Sobolev regularity of local well-posedness in deterministic versus probabilistic settings.}.

As it happens, a certain notion of criticality might not capture all the subtleties of the problem. As we shall see, some higher order iterates $u^{(n)}$ won't be better bounded than $u^{(n-1)}$ in the full subcritical range $T\ll \alpha^{-2}$ postulated above, but only in a subregion thereof. This is what defines our admissible blue region in Figure \ref{fig:admissible} above.

We should mention that the idea of utilizing the gain from randomness goes back to Bourgain \cite{Bourgain96} (in the random data setting) and Da Prato-Debussche \cite{DaPDe} (later, in the SPDE setting). They first noticed that the ansatz $u=u^{(0)}+\Rm$ allows one to put the remainder $\Rm$ in a higher regularity space than the linear term $u^{(0)}$. Their idea has since been applied to many different situations, see for example \cite{BB,BT,CO,Deng,DLM,KM,NS}, though most of these works either involve only the first order expansion (i.e. $N=0$), or involve higher order expansions with only suboptimal bounds (see for example \cite{BOP}). To the best of our knowledge, this is the first work where the sharp bounds for these $u^{(j)}$ terms are obtained to arbitrarily high order (at least in the dispersive setting).
\begin{rem}There are two main reasons why the high-order expansion (\ref{expansion1}) gives the sharp time of control, as opposed to previous works. The first one is because we are able to obtain sharp estimates for the terms $u^{(j)}$ with arbitrarily high order, which is not known previously due to the combinatorial complexity associated with trees (see Section \ref{treeintro} below).

The second reason is more intrinsic. In higher-order versions of the original Bourgain-Da Prato-Debussche approach, it usually stops improving in regularity beyond a certain point, due to the presence of the \emph{high-low} interactions (heuristically, the gain of powers of low frequency does not transform to the gain in regularity). This is a major difficulty in random data theory, and in recent years a few methods have been developed to address it, including the \emph{regularity structure} \cite{Hairer}, the \emph{para-controlled calculus} \cite{GIP}, and the \emph{random averaging operators} \cite{DNY2}. Fortunately, in the current problem this issue is absent, since the well-prepared initial data (\ref{wellprepared}) bounds the high-frequency components (where $|k|\sim 1$) and low-frequency components (where $|k|\sim L^{-1}$) uniformly, so the high-low interaction is simply controlled in the same way as the high-high interaction, allowing one to gain regularity indefinitely as the order increases.

\end{rem}
\subsubsection{Sharp estimates of Feynman trees}\label{treeintro} We start with the estimate for $u^{(n)}$. As is standard with the cubic nonlinear Schr\"{o}dinger equation, we first perform the Wick ordering by defining
\[w:=e^{-2i\lambda^2M_0t}\cdot u,\quad M_0:=\fint_{\mathbb{T}_L^d}|u|^2.\] Note that $M_0$ is essentially the mass which is conserved. Now $w$ satisfies the renormalized equation
\begin{equation}\label{neweqnintro}i\partial_tw-\Delta_\beta w+\lambda^2\bigg(|w|^2-2\fint_{\mathbb{T}_L^d}|w|^2\bigg)w=0,\end{equation} and $|\widehat{w_k}(t)|^2=|\widehat{u_k}(t)|^2$. This gets rid of the worst resonant term, which would otherwise lead to a suboptimal timescale.

Let $w^{(n)}$ be the $n$-th order iteration of the nonlinearity in (\ref{neweqnintro}), corresponding to the $u^{(n)}$ in (\ref{expansion1}). Since this nonlinearity is cubic, by induction it is easy to see that $w^{(n)}$ can be written (say in Fourier space) as a linear combination of terms\footnote{We will first perform rescaling and conjugation by the linear Schr\"{o}dinger flow, see Section \ref{reductions}; for simplicity we still use $\Jm_\Tm$ to denote these terms.} $\Jm_\Tm$, where $\Tm$ runs over all ternary trees with exactly
$n$ branching nodes (we will say it has \emph{scale} $\mathfrak s(\Tm)=n$). After some further reductions, the estimate for $\Jm_\Tm$ can be reduced to the estimate for terms of form\footnote{In reality one may have coefficients $m=m(k,k_1,\cdots,k_{2n+1})$ in the expression of $M_k$ in (\ref{multgauss}), but one can always reduce to the form of (\ref{multgauss}) by restricting to the level sets of $m$.}
\begin{equation}\label{multgauss}\Sigma_k:=\sum_{(k_1,\cdots,k_{2n+1})\in S}\eta_{k_1}^\pm\cdots \eta_{k_{2n+1}}^\pm, \qquad (\eta_k^+,\eta_k^{-}):=(\eta_k(\omega), \overline{\eta_k}(\omega)),
\end{equation} where $\eta_k(\omega)$ is as in \eqref{wellprepared}, $(k_1,\cdots,k_{2n+1})\in(\mathbb{Z}_L^d)^{2n+1}$, $S$ is a suitable finite subset of $(\mathbb{Z}_L^d)^{2n+1}$, and the $(2n+1)$ subscripts correspond to the $(2n+1)$ leaves of $\Tm$; see Definition \ref{tree1} and Figure \ref{fig:Trees}.

To estimate $M_k$ defined in (\ref{multgauss}) we invoke the standard large deviation estimates, see Lemma \ref{largedev}, which essentially asserts that $|\Sigma_k|\lesssim (\#S)^{1/2}$ with overwhelming probability, provided that there is no \emph{pairing} in $(k_1,\cdots,k_{2n+1})$, where a pairing $(k_i,k_j)$ means $k_i=k_j$ and the signs of $\eta_{k_i}$ and $\eta_{k_j}$ in (\ref{multgauss}) are the opposite. Moreover in the case of a pairing $(k_i,k_j)$ we can essentially replace $\eta_{k_i}^\pm \eta_{k_j}^\pm=|\eta_{k_i}|^2\approx 1$, so in general we can bound, with overwhelming probability, that
\[|\Sigma_k|^2\lesssim\sum_{(\mathrm{unpaired}\,\,\!k_i)}\bigg(\sum_{\substack{(\mathrm{paired}\,\,\!k_i):\\(k_1,\cdots,k_{2n+1})\in S}}1\bigg)^2\lesssim\sum_{(k_1,\cdots,k_{2n+1})\in S}1\cdot\sup_{(\mathrm{unpaired}\,\,\!k_i)}\sum_{\substack{(\mathrm{paired}\,\,\!k_i):\\(k_1,\cdots,k_{2n+1})\in S}}1.\] It thus suffices to bound the number of choices for $(k_1,\cdots,k_{2n+1})$ given the pairings, as well as the number of choices for the paired $k_j$'s given the unpaired $k_j$'s.

In the no pairing case such counting bounds are easy to prove, since the set $S$ is well adapted to the tree structure of $\Tm$; what makes the counting nontrivial is the pairings, especially the pairings between leaves that are faraway or from different levels (see Figure \ref{fig:Trees2}, where a pairing is depicted by an extra link between the two leaves). Nevertheless, we have developed a counting algorithm that specifically deals with the given pairing structure of $\Tm$, and ultimately leads to sharp counting bounds and consequently sharp bounds for $\Sigma_k$. See Proposition \ref{countingbd}.

\subsubsection{An $\ell^2$ operator norm bound} In contrast to the tree terms $\mathcal J_\mathcal T$ above, the remainder term $\Rm_{N+1}$ has no explicit random structure. Indeed, the only way it feels the ``chaos" of the initial data is through the equation it satisfies, which looks like (in integral form and in spatial Fourier variables)
$$
\Rm_{N+1}=\Jm_{\sim N}+\Lm(\Rm_{N+1}) +\Qm(\Rm_{N+1})+\Cm(\Rm_{N+1})
$$
where $\Jm_{\sim N}$ is a sum of Feynman trees $\Jm_\mathcal T$ (described above) of scale $\mathfrak s (\mathcal T)\sim N$, and $\Lm, \Qm,$ and $\Cm$ are respectively linear, bilinear, and trilinear operator in $\Rm_{N+1}$. The main point here is that one would like to propagate the estimates on $\Jm_{\sim N}$ discussed above to $\Rm_{N+1}$ itself; this is how we make rigorous the so-called ``propagation of chaos or quasi-gaussianity" claims that are often adopted in formal derivations.  

Since we're bootstrapping a smallness estimate on $\Rm_{N+1}$, any quadratic and cubic form of $\Rm_{N+1}$ will be easily bounded. As such, it suffices to propagate the bound for the term $\Lm(\Rm_{N+1})$, which reduces to bounding the $\ell^2\to \ell^2$ operator norm for the linear operator $\Lm$. By definition, the operator $\Lm$ will have the form $v\mapsto \mathcal{IW}(\Jm_{\mathcal T_1}, \Jm_{\mathcal T_2}, v)$ where $\mathcal{I}$ is the Duhamel operator, $\mathcal{W}$ is the trilinear form coming from the cubic nonlinearity, and $\Jm_{\mathcal T_1}, \Jm_{\mathcal T_2}$ are trees of scale $\leq N$; thus in Fourier space it can be viewed as a matrix with \emph{random coefficients}. The key to obtaining the sharp estimate for $\Lm$ is then to exploit the cancellation coming from this randomness, and the most efficient way to do this is via the $TT^*$ method.

In fact, the idea of applying $TT^*$ method to random matrices has already been used in Bourgain \cite{Bourgain96}. In that paper one is still far above (the probabilistic) criticality, so applying the $TT^*$ method once already gives adequate control. In the current case, however, we are aiming at obtaining sharp estimates, so applying $TT^*$ once will not be sufficient.

The solution is thus to apply $TT^*$ sufficiently many (say $D\gg 1$) times, which leads to the analysis of the kernel of the operator $(\Lm\Lm^*)^D$. At first sight this kernel seems to be a complicated multilinear expression which is difficult to handle; nevertheless we make one key observation, namely this kernel can essentially be recast in the form of (\ref{multgauss}), for some large auxiliary tree $\Tm=\Tm^D$, which is obtained from a single root node by attaching copies of the trees $\mathcal T_1$ and $\mathcal T_2$ successively, for a total of $2D$ times (see Figure \ref{fig:manytrees}). With this observation, the arguments in the previous section then lead to sharp bounds of the kernel of $(\Lm\Lm^*)^D$, up to some loss that is a power of $L$ independent of $D$; upon taking the $1/(2D)$ power and choosing $D$ sufficiently large, this power will become negligible and imply the sharp bound for the operator norm of $\Lm$. See Section \ref{operatorsec}.

\subsubsection{Sharpness of estimates}\label{sharpness} Finally we remark that, the estimates we prove for $\Jm_\Tm$ are sharp up to some finite power of $L$ (independent of $\Tm$). More precisely, from Proposition \ref{lwp1} we know that for any ternary tree $\Tm$ of scale $n$ and possible pairing structure (see Definition \ref{dfpairing}), with overwhelming probability 
\begin{equation}\label{upperbd}\sup_k\|(\Jm_\Tm)_k\|_{h^{b}}\leq L^{0+}\rho^n,
\end{equation} where $\rho$ is some quantity depending on $\alpha$, $L$ and $T$ (see (\ref{defrho})), $k$ is the spatial Fourier variable and $h^b$ is a time-Sobolev norm defined in (\ref{hb}); on the other hand, we will show that that for some particular choice of trees $\Tm$ of scale $n$, and some particular choice of pairings, with high probability
\begin{equation}\label{lowerbd}\sup_k\|(\Jm_\Tm)_k\|_{h^{b}}\geq L^{-d}\rho^n.
\end{equation}The timescale $T$ of Theorem \ref{thm2} is the largest that makes $\rho\ll 1$; thus if one wants to go beyond $T$ in cases other than Theorem \ref{thm1}, it would be necessary to address the divergence of (\ref{lowerbd}) with $\rho\gg 1$ by exploiting the cancellation between different tree terms and/or different pairing choices. See Section \ref{badterm}.
\subsection{Organization of the paper} In Section \ref{trees and LTE}, we explain the diagrammatic expansion of the solution into Feynman trees, and state the a priori estimates on such trees and remainder term, which yield the long-time existence of such expansions. Section \ref{mainproof} is devoted to the proof of those a priori estimates. In Section \ref{proof of theorem}, we prove the main theorems mentioned above, and in the Section \ref{number theory section}, we prove the needed number theoretic results that allow to replace the highly-oscillating Riemann sums by integrals. 
 
\subsection{Notations}\label{notat} Most notations will be standard. Let $z^+=z$ and $z^-=\overline{z}$. Define $|k|_\beta$ by $|k|_\beta^2=\beta_1k_1^2+\cdots+\beta_dk_d^2$ for $k=(k_1,\cdots,k_d)$. The spatial Fourier series of a function $u: \T_L^d \to \mathbb C$ is defined on $\Z^d_L:=L^{-1}\Z^{d}$ by
\begin{equation}\label{fourierset}
\widehat{u}_k=\int_{\T^d_L} u(x) e^{-2\pi i k\cdot x},\quad \mathrm{\; so \,that \;}\quad u(x)=\frac{1}{L^d}\sum_{k \in \Z^d_L} \widehat{u}_k \,e^{2\pi i k\cdot x}. 
\end{equation}

The temporal Fourier transform is defined by
\[\widetilde{f}(\tau)=\int_{\mathbb{R}}e^{-2\pi it\tau}f(t)\,\mathrm{d}t.\]

{Let $\delta>0$ be fixed throughout the paper.} Let $N$ and $s$ and $b>\frac{1}{2}$ be fixed, such that $N$ and $s$ are large enough, and $b-\frac{1}{2}$ is small enough, depending on $d$ and $\delta$. The quantity $C$ will denote any large absolute constant (which does not depend on $(N,s,b-\frac{1}{2})$), and $\theta$ will denote any small positive constant (which depends on $(N,s,b-\frac{1}{2})$); these may change from line to line. The symbols $O(\cdot)$ and $\lesssim$ etc. will have their usual meanings, with implicit constants depending on $\theta$. Let $L$ be large enough depending on all the above implicit constants. If some statement $S$ involving $\omega$ is true with probability $\geq 1-Ke^{-L^\theta}$ for some constant $K$ (depending on $\theta$), then we say this statement $S$ is $L$-certain.

When a function depends on many variables, we may use notions like
\[f=f(x_i:i\in A,\,y_j:1\leq j\leq m)\] to denote a function $f$ of variables $(x_i:i\in A)$ and $y_1,\cdots,y_m$.

  \section{Tree expansions and long-time existence}\label{trees and LTE}
 \subsection{First reductions}\label{reductions} Let $\widehat{u}_k(t)$ be the Fourier coefficients of $u(t)$, as in (\ref{fourierset}).
 Then with $c_k(t):= e^{2\pi i|k|_\beta^2t} \widehat u_k(t)=(\mathcal F_{\T^d_L} e^{-it\Delta_\beta} u)(k)$, we arrive at the following equation for the Fourier modes:
 \begin{equation}\label{akeqn}
\begin{cases}
 i \dot{c_k} =  \left(\frac{\lambda}{L^{d}}\right)^{2} \sum\limits_{\substack{(k_1,k_2,k_{3}) \in (\mathbb{Z}^d_L)^3 \\ k - k_1 + k_2 -k_3 = 0}} c_{k_1}\overline{c_{k_2}}  c_{k_3} e^{2\pi i \Omega(k_1,k_2,k_3,k)t} \\[3em]
c_k(0) = (c_k)_{\mathrm{in}}=\widehat u_k(0)
\end{cases}
\end{equation}
where 
$
\Omega(k_1,k_2,k_3,k) =|k_1|_\beta^2-|k_2|_\beta^2+|k_3|_\beta^2-|k|_\beta^2.
$
Note that the sum above can be written as
$$
\sum\limits_{\substack{(k_1,k_2, k_{3}) \in (\mathbb{Z}^d_L)^3 \\ k - k_1 + k_2 -k_3 = 0}}=2\sum_{k_1=k}-\sum_{k_1=k_2=k_3} +\sum_{k_1, k_3\neq k}
$$
which allows to write, defining $M=\sum_{k_3} |c_{k_3}|^2$ (which is conserved),
$$
 i \dot{c_k} =  \bigg(\frac{\lambda}{L^{d}}\bigg)^{2} 
\bigg(2M c_k -|c_k|^2c_k+ \sum\limits_{(k_1,k_2, k_{3})}^{\times} c_{k_1}\overline{c_{k_2}}  c_{k_3} e^{2\pi i \Omega(k_1,k_2,k_3,k)t}\bigg).
$$ Here and below $\sum^{\times}$ represents summation under the conditions $k_j\in\mathbb{Z}_L^d$, $k_1-k_2+k_3=k$ and $k\not\in\{k_1,k_3\}$. Introducing $b_k(t)=c_k(t)e^{-2i(L^{-d}\lambda)^{2}Mt}$, we arrive at the following equation for $b_k(t)$:
 \begin{equation}\label{bkeqn}
\begin{cases}
i \dot{b_k} =  \left(\frac{\lambda}{L^{d}}\right)^{2} 
\bigg( -|b_k|^2b_k+ \sum\limits_{(k_1,k_2, k_{3})}^{\times} b_{k_1}\overline{b_{k_2}}  b_{k_3} e^{2\pi i \Omega(k_1,k_2,k_3,k)t}\bigg),\\[3em]
b_k(0) = (b_k)_{\mathrm{in}}=\widehat u_k(0).
\end{cases}
\end{equation}

In Theorem \ref{thm2} we will be studying the solution $u(t)$, or equivalently the sequence $(b_k(t))_{k \in \Z^d_L}$, on a time interval $[0,T]$. It will be convenient, to simplify some notation later, to work on the unit time interval $[0,1]$. For this we introduce the final ansatz
$$
a_k(t)=b_k(Tt),
$$ 
which satisfies the following equation
 \begin{equation}\label{akeqn}
\begin{cases}
i \dot{a_k} =  \left(\frac{\alpha T}{L^{d}}\right)
\bigg( -|a_k|^2a_k+ \sum\limits_{(k_1,k_2, k_{3})}^{\times} a_{k_1}\overline{a_{k_2}}  a_{k_3} e^{2\pi i T\Omega(k_1,k_2,k_3,k)t}\bigg),\\[3em]
a_k(0) =(a_k)_{\mathrm{in}}=\widehat u_k(0).
\end{cases}
\end{equation}
Here we have also used the relation $\alpha =\lambda^2L^{-d}$. Recall the well-prepared initial data (\ref{wellprepared}), which transforms into the initial data for $a_k$,
\begin{equation}\label{initial0}
(a_k)_{\mathrm{in}}=\sqrt{n_{\mathrm{in}}} \cdot\eta_{k}(\omega)
\end{equation}
where $\eta_{k}(\omega)$ are the same as in (\ref{wellprepared}).

\subsection{The tree expansion} Let $\boldsymbol a(t) =(a_k(t))_{k \in \Z^d_L}$ and $\boldsymbol{a}_{\mathrm{in}} =\boldsymbol a(0)$. Let $J=[0,1]$, we will fix a smooth compactly supported cutoff function $\chi$ such that $\chi\equiv 1$ on $J$. Then by (\ref{akeqn}), we know that for $t\in J$ we have
\begin{equation}\label{eqnnewa}\boldsymbol{a}(t)=\chi(t)\boldsymbol{a}_{\mathrm{in}}+\mathcal{IW}(\boldsymbol a, \boldsymbol a, \boldsymbol a)(t)\end{equation} where the Duhamel term is defined by
\begin{equation}\label{Duhamel00}
\mathcal{I}F(t)=\chi(t)\int_0^t\chi(t')F(t')\,\mathrm{d}t',
\end{equation}
\begin{equation}\label{Duhamel0}
\mathcal{W}(\boldsymbol b, \boldsymbol c , \boldsymbol d)_k(t):=-\frac{i\alpha T}{L^{d}}\bigg( -(b_{k}\overline{c_{k}}d_{k})(t)+ \sum\limits_{(k_1,k_2, k_{3})}^{\times} (b_{k_1}\overline{c_{k_2}}  d_{k_3})(t) e^{2\pi i T\Omega(k_1,k_2,k_3,k)t}\bigg).
\end{equation}

Since we will only be studying $\boldsymbol{a}$ for $t\in J$, so from now on \emph{we will replace $\boldsymbol{a}$ by the solution to (\ref{eqnnewa}) for $t\in\mathbb{R}$} (the existence and uniqueness of the latter will be clear from the proof below). We will be analyzing the temporal Fourier transform of this (extended) $\boldsymbol{a}$, so let us first record a formula for $\mathcal{I}$ on the Fourier side.
\begin{lem}\label{duhamfourier} Let $\mathcal{I}$ be defined in (\ref{Duhamel00}), then we have (recall that $\widetilde{G}$ means the temporal Fourier transform of $G$)
\begin{equation}\label{duhamelbound}\widetilde{\mathcal{I}F}(\tau)=\int_{\mathbb{R}}(I_0+I_1)(\tau,\sigma)\widetilde{F}(\sigma)\,\mathrm{d}\sigma,\quad|\partial_{\tau,\sigma}^aI_d(\tau,\sigma)|\lesssim_{a,A}\frac{1}{\langle \tau-d\sigma\rangle^A}\frac{1}{\langle\sigma\rangle}.\end{equation}
\end{lem}
\begin{proof} See \cite{DNY}, Lemma 3.1.
\end{proof}
Now define $\Jm_n$ recursively by
\begin{equation}\label{defjn}
\begin{split}
\Jm_0(t)&=\chi(t)\cdot\boldsymbol{a}_{\mathrm{in}},\\
\Jm_n(t)&=\sum_{n_1+n_2+n_3=n-1}\mathcal{IW}(\Jm_{n_1},\Jm_{n_2},\Jm_{n_3})(t),
\end{split}
\end{equation} and define
\begin{equation}\label{defrn+1}\Jm_{\leq N}=\sum_{n\leq N}\Jm_{n},\quad \mathcal{R}_{N+1}=\boldsymbol{a}-\Jm_{\leq N}.\end{equation} By plugging in \eqref{eqnnewa}, we get that $\Rm_{N+1}$ satisfies the equation
\begin{equation}\label{eqnr}\Rm_{N+1}=\Jm_{\sim N}+\Lm(\Rm_{N+1}) +\Qm(\Rm_{N+1})+\Cm(\Rm_{N+1}),
\end{equation} where the relevant terms are defined as
\begin{align}
\Jm_{\sim N}&:=\sum_{\substack{n_1, n_2, n_3\leq N\\  n_1+n_2+n_3\geq N}} \mathcal{IW}(\Jm_{n_1},\Jm_{n_2},\Jm_{n_3}), \\
\Lm(v)&:=\sum_{n_1,n_2\leq N}\big(2\mathcal{IW}(\Jm_{n_1}, \Jm_{n_2}, v)+\mathcal{IW}(\Jm_{n_1}, v, \Jm_{n_2})\big),  \\
\Qm(v)&:=\sum_{n_1\leq N}\big(2\mathcal{IW}(v,v, \Jm_{n_1})+\mathcal{IW}(v, \Jm_{n_1}, v)\big),\\
\Cm(v)&:=\mathcal{IW}(v, v, v).
\end{align} 

Next we will derive a formula for the time Fourier transform of $\Jm_n$; for this we need some preparations regarding multilinear forms associated with ternary trees. 
\begin{df}\label{tree1} 
\begin{enumerate}
\item Let $\mathcal{T}$ be a ternary tree. We use $\mathcal{L}$ to denote the set of leaves and $l$ their number, $\mathcal{N}=\Tm\backslash\Lm$ the set of branching nodes and $n$ their number, and $\mathfrak{r} \in \mathcal N$ the root node. The \emph{scale} of a ternary tree $\mathcal T$ is defined as $\mathfrak s(\mathcal T)=n$ (the number of branching nodes)\footnote{By convention, the scale of a single node is 0.}. A tree of scale $n$ has $l=2n+1$ leaves, and a total of $3n+1$ vertices. 

\item (Signs on a tree) For each node $\mathfrak{n}\in\Nm$ let its children from left to right be $\mathfrak{n}_1$, $\mathfrak{n}_2$, $\mathfrak{n}_3$. We fix the sign $\iota_{\mathfrak{n}}\in\{\pm\}$ as follows: first $\iota_{\mathfrak{r}}=+$, then for any node $\mathfrak{n}\in \Nm$, define $\iota_{\mathfrak{n}_1}=\iota_{\mathfrak{n}_3}=\iota_{\mathfrak{n}}$ and $\iota_{\mathfrak{n}_2}=-\iota_{\mathfrak{n}}$. 

\begin{figure}[h!]
  \includegraphics[width=.8\linewidth]{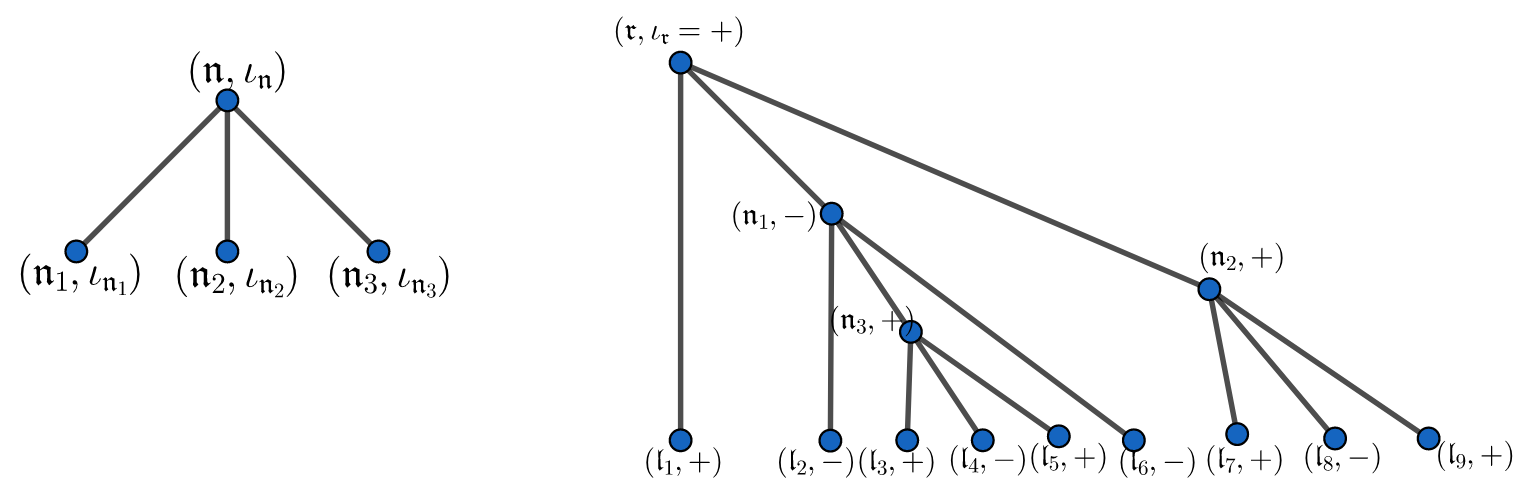}
  \caption{On the left, a node $\mathfrak n$ with its three children $\mathfrak n_1, \mathfrak n_2, \mathfrak n_3$ with signs $\iota_1=\iota_3=\iota=-\iota_2$. On the right, a tree of scale four ($\mathfrak s(\mathcal T)=4)$ with root $\mathfrak r$, four branching nodes ($\mathfrak r, \mathfrak n_1, \mathfrak n_2, \mathfrak n_3$), and $l=9$ leaves along with their signatures.}
  \label{fig:Trees}
\end{figure}

\item (Admissible assignments) Suppose we assign to each $\mathfrak{n}\in\mathcal{T}$ an element $k_{\mathfrak{n}}\in\mathbb{Z}_L^d$. We say such an assignment $(k_{\mathfrak{n}}:\mathfrak{n}\in\mathcal{T})$ is \emph{admissible} if for any $\mathfrak{n}\in\Nm$ we have $k_{\mathfrak{n}}=k_{\mathfrak{n}_1}-k_{\mathfrak{n}_2}+k_{\mathfrak{n}_3}$, and that either $k_{\mathfrak{n}}\not\in\{k_{\mathfrak{n}_1},k_{\mathfrak{n}_3}\}$ or $k_{\mathfrak{n}}=k_{\mathfrak{n}_1}=k_{\mathfrak{n}_2}=k_{\mathfrak{n}_3}$. Clearly an admissible assignment is completely determined by the values of $k_{\mathfrak{l}}$ for $\mathfrak{l}\in\mathcal{L}$. For any assignment we denote $\Omega_{\mathfrak{n}}:=\Omega(k_{\mathfrak{n}_1},k_{\mathfrak{n}_2},k_{\mathfrak{n}_3},k_{\mathfrak{n}})$; suppose we also fix\footnote{This assignment is arbitrary, but will usually be omitted since there are finitely many choices.} $d_{\mathfrak{n}}\in\{0,1\}$ for each $\mathfrak{n}\in\Nm$, we can define $q_{\mathfrak{n}}$ for each $\mathfrak{n}\in\Tm$ inductively by
\begin{equation}\label{defqn}q_{\mathfrak{n}}=0\mathrm{\ if\ }\mathfrak{n}\in\Lm;\quad q_{\mathfrak{n}}=d_{\mathfrak{n}_1}q_{\mathfrak{n}_1}-d_{\mathfrak{n}_2}q_{\mathfrak{n}_2}+d_{\mathfrak{n}_3}q_{\mathfrak{n}_3}+\Omega_{\mathfrak{n}}\mathrm{\ if\ }\mathfrak{n}\in\Nm.\end{equation}
\end{enumerate}
\end{df}

\begin{prop}\label{treedefs} For each ternary tree $\mathcal{T}$ define $\Jm_{\mathcal{T}}$ inductively by
\begin{equation}\label{defjt}
\Jm_{\bullet}(t)=\chi(t)\cdot\boldsymbol{a}_{\mathrm{in}},\quad
\Jm_{\mathcal{T}}(t)=\mathcal{IW}(\Jm_{\Tm_1},\Jm_{\Tm_2},\Jm_{\Tm_3})(t),
\end{equation} where $\bullet$ represents the tree with a single node, and $\Tm_1$, $\Tm_2$, $\Tm_3$ are the subtrees rooted at the three children of the root node of $\Tm$. Then we have
\begin{equation}\label{jttojn}\Jm_n=\sum_{\mathfrak s(\mathcal T)=n}\Jm_{\Tm}.
\end{equation} Moreover for any $\Tm$ of scale $\mathfrak s(\mathcal T)=n$ we have the formula
\begin{equation}\label{formulajt}(\widetilde{\Jm_{\Tm}})_{k}(\tau)=\bigg(\frac{\alpha T}{L^{d}}\bigg)^n\sum_{(k_{\mathfrak{n}}:\mathfrak{n}\in\mathcal{T})}\mathcal{K}_{\Tm}(\tau,k_{\mathfrak{n}}:\mathfrak{n}\in\Tm)\prod_{\mathfrak{l}\in\Lm}\sqrt{n_{\mathrm{in}}(k_{\mathfrak{l}})}\cdot\prod_{\mathfrak{l}\in\Lm}\big[\eta_{k_{\mathfrak{l}}}(\omega)\big]^{\iota_{\mathfrak{l}}},
\end{equation} where the sum is taken over all admissible assignments $(k_{\mathfrak{n}}:\mathfrak{n}\in\mathcal{T})$ such that $k_{\mathfrak{r}}=k$, and the function $\mathcal{K}=\mathcal{K}_{\Tm}(\tau,k_{\mathfrak{n}}:\mathfrak{n}\in\Tm)$ satisfies
\begin{equation}\label{boundk}|\partial_\tau^a\mathcal{K}|\lesssim_{a,A}\sum_{(d_{\mathfrak{n}}:\mathfrak{n}\in\Nm)}\langle \tau-Td_{\mathrm{r}}q_{\mathrm{r}}\rangle^{-A}\cdot\prod_{\mathfrak{n}\in\Nm}\langle Tq_{\mathfrak{n}}\rangle^{-1},
\end{equation} where $q_{\mathfrak{n}}$ is defined in (\ref{defqn}).
\end{prop}
\begin{proof} First (\ref{jttojn}) follows from the definitions (\ref{defjn}) and (\ref{defjt}) and an easy induction. We now prove (\ref{formulajt}) and (\ref{boundk}) inductively, noting also that $(a_k)_{\mathrm{in}}=\sqrt{n_{\mathrm{in}}(k)}\cdot\eta_k(\omega)$. For $\mathcal{T}=\bullet$ (\ref{formulajt}) follows from (\ref{defjt}) with $\mathcal{K}_{\Tm}(\tau,k_{\mathfrak{r}})=\widetilde{\chi}(\tau)$ that satisfies (\ref{boundk}). Now suppose (\ref{formulajt}) and (\ref{boundk}) are true for smaller trees, then by (\ref{Duhamel0}), (\ref{defjt}) and Lemma \ref{duhamfourier}, up to unimportant coefficients, we can write
\[(\widetilde{\Jm_{\Tm}})_{k}(\tau)=\frac{i\alpha T}{L^{d}}\sum_{d\in\{0,1\}}\sum_{(k_1,k_2,k_3)}^*\int_{\mathbb{R}^3}I_d(\tau,\sigma)\prod_{j=1}^3\big[(\widetilde{\Jm_{\Tm_j}})_{k_j}(\tau_j)\big]^{\iota_j}\,\mathrm{d}\tau_j,\] where $\sum^*$ represents summation under the conditions $k_j\in \mathbb{Z}_L^d$, $k_1-k_2+k_3=k$ and either $k\not\in\{k_1,k_3\}$ or $k=k_1=k_2=k_3$, the signs $(\iota_1,\iota_2,\iota_3)=(+,-,+)$, and $\sigma=\tau_1-\tau_2+\tau_3+T\Omega(k_1,k_2,k_3,k)$. Now applying the induction hypothesis, we can write $(\widetilde{\Jm_{\Tm}})_{k}(\tau)$ in the form of (\ref{formulajt}) with the function 
\begin{equation}\label{function2}\mathcal{K}_{\Tm}(\tau,k_{\mathfrak{n}}:\mathfrak{n}\in\Tm)=\sum_{d\in\{0,1\}}\int_{\mathbb{R}^3}I_d(\tau,\tau_1-\tau_2+\tau_3+T\Omega_{\mathfrak{r}})\prod_{j=1}^3\big[\mathcal{K}_{\Tm_j}(\tau_j,k_{\mathfrak{n}}:\mathfrak{n}\in\Tm_j)\big]^{\iota_j}\,\mathrm{d}\tau_j,
\end{equation} where $\mathfrak{r}$ is the root of $\mathcal{T}$ with children $\mathfrak{r}_1,\mathfrak{r}_2,\mathfrak{r}_3$ and $\Tm_j$ is the subtree rooted at $\mathfrak{r}_j$.

It then suffices to prove that $\mathcal{K}_{\Tm}$ defined by (\ref{function2}) satisfies (\ref{boundk}). By induction hypothesis, we may fix a choice $d_{\mathfrak{n}}$ for each non-leaf node $\mathfrak{n}$ of each $\Tm_j$, and let $d_{\mathfrak{r}}=d$. Then plugging (\ref{boundk}) into (\ref{function2}) we get that
\[|\partial_\tau^a\mathcal{K}_\Tm|\lesssim_{a,A}\prod_{\mathfrak{r}\neq\mathfrak{n}\in\Nm}\frac{1}{\langle T q_{\mathfrak{n}}\rangle}\int_{\mathbb{R}^3}\frac{1}{\langle\tau-d(\tau_1-\tau_2+\tau_3+T\Omega_{\mathfrak{r}})\rangle^A}\frac{1}{\langle \tau_1-\tau_2+\tau_3+T\Omega_{\mathfrak{r}}\rangle}\prod_{j=1}^3\frac{\mathrm{d}\tau_j}{\langle\tau_j-Td_{\mathfrak{r}_j}q_{\mathfrak{r}_j}\rangle^A},\] which upon integrating in $\tau_j$ gives (\ref{boundk}). This completes the proof.
\end{proof}
\subsection{Statement of main estimates}\label{state} Define the $h^b$ space by
\begin{equation}\label{hb}
\|a(t)\|_{h^{b}} = \bigg(\int_{\R} \langle \tau \rangle^{2b}|\widetilde{a}(\tau)|^2 \, d\tau\bigg)^{\frac12},
\end{equation} and similarly the $h^{s,b}$ space for $\boldsymbol a(t)=(a_k(t))_{k \in \Z^d_L}$
\begin{equation}\label{hsb}
\|\boldsymbol a\|_{h^{s,b}} = \bigg(L^{-d}\sum_{k \in \Z^d_L}\int_{\R} \langle \tau \rangle^{2b}\langle k\rangle^{2s}|\widetilde{a}_k(\tau)|^2 \, d\tau\bigg)^{\frac12}.
\end{equation} We shall estimate the solution $u$ in an appropriately rescaled $X^{s, b}$ space, which is equivalent to estimating the sequence $\boldsymbol a(t)=\left(a_k(t)\right)_{k \in \Z^d_L}$ in the space $h^{s, b}$. Define the quantity \begin{equation}\label{defrho}\rho:=\left\{
\begin{aligned}&\alpha T,&&{\rm{if\ }}1\leq T\leq L;\\
&\alpha L,&&{\rm{if\ }}L\leq T\leq L^2;\\
&\alpha TL^{-1},&&{\rm{if\ }} T\geq L^2{\rm\ and\ }\beta_i{\rm\ generic}.
\end{aligned}
\right.\end{equation} 
By the definition of $\delta>0$ in \eqref{adm1}, we can verify that $\alpha T^{1/2}\leq\rho\leq L^{-\delta}$. 

\medskip

\begin{prop}[Well-posedness bounds]\label{lwp0}
Let $\rho$ be defined in (\ref{defrho}), then $L$-certainly, for all $1\leq n\leq 3N$, we have
\begin{equation}\label{lwp0-1}\sup_k\langle k\rangle^{2s}\|(\mathcal{J}_n)_k\|_{h^{b}}\leq L^{\theta+C(b-\frac{1}{2})}\rho^{n-1}(\alpha T^{\frac{1}{2}})\leq L^{\theta+C(b-\frac{1}{2})}\rho^{n},
\end{equation}
\begin{equation}\label{lwp2-1}\|\mathcal{R}_{N+1}\|_{h^{s,b}}\leq \rho^{N}.
\end{equation} 
\end{prop} Proposition \ref{lwp0} follows from the following two bounds, which will be proved in Section \ref{mainproof}.
\begin{prop}[Bounds of tree terms]\label{lwp1} We have $L$-certainly that
\begin{equation}\label{lwp1-1}\sup_k\langle k\rangle^{2s}\|(\mathcal{J}_\Tm)_k\|_{h^{b}}\leq L^{\theta+C(b-\frac{1}{2})}\rho^{n-1}(\alpha T^{\frac{1}{2}})\leq L^{\theta+C(b-\frac{1}{2})}\rho^{n}
\end{equation} for any ternary tree of depth $n$, where $0\leq n\leq 3N$.
\end{prop}
\begin{prop}[An operator norm bound]\label{lwp3} We have $L$-certainly that, for any trees $\Tm_1,\Tm_2$ with $|\Tm_j|=3n_j+1$ and $0\leq n_1,n_2\leq N$, the operators
\begin{equation}\label{defop}\mathcal{P}_+:v\mapsto\mathcal{IW}(\Jm_{\Tm_1},\Jm_{\Tm_2},v)\quad\mathrm{and}\quad \mathcal{P}_-:v\mapsto\mathcal{IW}(\Jm_{\Tm_1},v,\Jm_{\Tm_2})\end{equation} satisfy the bounds
\begin{equation}\label{boundop}\|\mathcal{P}_{\pm}\|_{h^{s,b}\to h^{s,b}}\leq L^{\theta}\rho^{n_1+n_2+\frac{1}{2}}.
\end{equation}
\end{prop}
\begin{proof}[Proof of Proposition \ref{lwp0} assuming Propositions \ref{lwp1} and \ref{lwp3}] Assume we have already excluded an exceptional set of probability $\lesssim e^{-L^\theta}$. The bound (\ref{lwp0-1}) follows directly from (\ref{jttojn}) and (\ref{lwp1-1}), it remains to bound $\mathcal{R}_{N+1}$. Recall that $\mathcal{R}_{N+1}$ satisfies the equation (\ref{eqnr}), so it suffices to prove that the mapping
\[v\mapsto \Jm_{\sim N}+\Lm(v)+\Qm(v)+\Cm(v)\] is a contraction mapping from the set $\mathcal{Z}=\{v:\|v\|_{h^{s,b}}\leq \rho^{N}\}$ to itself. We will only prove that it maps $\mathcal{Z}$ into $\mathcal{Z}$ as the contraction part follows in a similar way; now suppose $\|v\|_{h^{s,b}}\leq \rho^N$, then by (\ref{jttojn}) and (\ref{lwp1-1}) we have
\[ \|\Jm_{\sim N}\|_{h^{s,b}}^2\sim L^{-d}\sum_{k\in\mathbb{Z}_L^d}\langle k\rangle^{2s}\|(\Jm_{\sim N})_k\|_{h^b}^2\lesssim (L^{\theta+C(b-\frac{1}{2})}\rho^{N+1})^2\cdot L^{-d}\sum_{k\in\mathbb{Z}_L^d}\langle k\rangle^{-2s}\ll\rho^{2N},\] so $\|\Jm_{\sim N}\|_{h^{s,b}}\ll \rho^N$. Next we may use (\ref{boundop}) to bound
\[\|\Lm(v)\|_{h^{s,b}}\leq L^\theta\rho^{\frac{1}{2}}\cdot\|v\|_{h^{s,b}}\leq L^\theta\rho^{\frac{1}{2}}\cdot \rho^N\ll\rho^N.\] As for the terms $\Qm(v)$ and $\Cm(v)$, we apply the simple bound (which easily follows from (\ref{Duhamel0}))
\begin{multline}\|\mathcal{IW}(u,v,w)\|_{h^{s,b}}\lesssim\|\mathcal{IW}(u,v,w)\|_{h^{s,1}}\lesssim\|\mathcal{IW}(u,v,w)\|_{h^{s,0}}+\|\partial_t\mathcal{IW}(u,v,w)\|_{h^{s,0}}\\\lesssim\frac{\alpha T}{L^{d}}\sum_{\mathrm{cyc}}\|u\|_{h^{s,0}}\|v_k(t)\|_{\ell_k^1L_t^\infty}\|w_k(t)\|_{\ell_k^1L_t^\infty}\lesssim \alpha TL^d\|u\|_{h^{s,b}}\|v\|_{h^{s,b}}\|w\|_{h^{s,b}},\end{multline} where $\sum_{\mathrm{cyc}}$ means summing in permutations of $(u,v,w)$. As $\alpha T\leq L^{d}$, we conclude (using also Proposition \ref{lwp1}) that
\[\|\Qm(v)\|_{h^{s,b}}+\|\Cm(v)\|_{h^{s,b}}\lesssim \alpha TL^{\theta+d+C(b-\frac{1}{2})}\rho^{2N}\ll\rho^N,\] since $\rho\leq L^{-\delta}$ and $N\gg\delta^{-1}$. This completes the proof.
\end{proof}

    \section{Proof of main estimates}\label{mainproof} In this section we prove Propositions \ref{lwp1} and \ref{lwp3}.
  \subsection{Large deviation and basic counting estimates}  We start by making some preparations, namely the large deviation and counting estimates that will be used repeatedly in the proof below.
  \begin{lem}\label{largedev}
Let $\{\eta_k(\omega)\}$ be i.i.d. complex random variables, such that the law of each $\eta_k$ is either Gaussian with mean $0$ and variance $1$, or the uniform distribution on the unit circle. Let $F=F(\omega)$ be defined by \begin{equation}\label{indp}F(\omega)=\sum_{k_1,\cdots,k_n}a_{k_1\cdots k_n}\prod_{j=1}^n\eta_{k_j}^{\iota_j},
\end{equation} where $a_{k_1\cdots k_n}$ are constants, then $F$ can be divided into finitely many terms, and for each term there is a choice of $X=\{i_1,\cdots,i_p\}$ and $Y=\{j_1,\cdots,j_p\}$ which are two disjoint subsets of $\{1,2,\cdots,n\}$, such that\begin{equation}\label{largedevest}\mathbb{P}(|F(\omega)|\geq A\cdot M^{\frac{1}{2}})\leq Ce^{-cA^{\frac{2}{n}}}
\end{equation} holds with
\begin{equation}\label{ortho}M=\sum_{(k_{\ell}):\ell\not\in X\cup Y}\bigg(\sum_{\mathrm{pairing}\,(k_{i_s},k_{j_s}):1\leq s\leq p}|a_{k_1\cdots k_n}|\bigg)^2,
\end{equation} where a pairing $(k_{i},k_{j})$ means $(\iota_i+\iota_j,\iota_ik_i+\iota_jk_j)=0$.
\end{lem}
\begin{proof} First assume $\eta_k$ is Gaussian. Then by the standard hypercontractivity estimate for Ornstein-Uhlenbeck semigroup {(see for example \cite{OT})}, we know that (\ref{largedevest}) holds with $M$ replaced by $\mathbb{E}|F(\omega)|^2$. Now to estimate $\mathbb{E}|F(\omega)|^2$, by dividing the sum (\ref{indp}) into finitely many terms and rearranging the subscripts, we may assume in a monomial of (\ref{indp}) that
\begin{equation}k_{1}=\cdots=k_{j_1},\,\, k_{j_1+1}=\cdots=k_{j_2},\cdots,k_{j_{r-1}+1}=\cdots=k_{j_r},\,\,1\leq j_1<\cdots <j_r=n,
\end{equation} and the $k_{j_s}$ are different for $1\leq s\leq r$. Such a monomial has the form
\[\prod_{s=1}^r\eta_{k_{j_s}}^{b_s}(\overline{\eta_{k_{j_s}}})^{c_s},\quad b_s+c_s=j_s-j_{s-1}\,\,(j_0=0),\] where the factors for different $s$ are independent. We may also assume $b_s=c_s$ for $1\leq s\leq q$ and $b_s\neq c_s$ for $q+1\leq s\leq r$, and for $1\leq j\leq j_q$ we may assume $\iota_j$ has the same sign as $(-1)^j$. Then we can further rewrite this monomial as a linear combination of
\[\prod_{s=1}^pb_s!\prod_{s=p+1}^q(|\eta_{k_{j_s}}|^{2b_s}-b_s!)\prod_{s=q+1}^r\eta_{k_{j_s}}^{b_s}(\overline{\eta_{k_{j_s}}})^{c_s}\] for $1\leq p\leq q$. Therefore, $F(\omega)$ is a finite linear combination of expressions of form
\[\sum_{k_{j_1},\cdots,k_{j_r}}a_{k_{j_1},\cdots,k_{j_1},\cdots,k_{j_r},\cdots k_{j_r}}\prod_{s=1}^pb_s!\prod_{s=p+1}^q(|\eta_{k_{j_s}}|^{2b_s}-b_s!)\prod_{s=q+1}^r\eta_{k_{j_s}}^{b_s}(\overline{\eta_{k_{j_s}}})^{c_s}.\] Due to independence and the fact that $\mathbb{E}(|\eta|^{2b}-b!)=\mathbb{E}(\eta^b(\overline{\eta})^c)=0$ for a normalized Gaussian $\eta$ and $b\neq c$, we conclude that
\begin{equation}\mathbb{E}|F(\omega)|^2\lesssim\sum_{k_{j_{p+1}},\cdots k_{j_r}}\bigg(\sum_{k_{j_1},\cdots,k_{j_p}}|a_{k_{j_1},\cdots,k_{j_1},\cdots,k_{j_r},\cdots k_{j_r}}|\bigg)^2,
\end{equation} which is bounded by the right hand side of (\ref{ortho}), by choosing $X=\{1,3,\cdots,j_p-1\}$ and $Y=\{2,4,\cdots,j_p\}$, as under our assumptions $(k_{2i-1},k_{2i})$ is a pairing for $2i\leq j_p$.

Now assume $\eta_k$ is uniformly distributed on the unit circle. Let $\{g_k(\omega)\}$ be i.i.d. normalized Gaussians as in the first part, and consider the random variable
\[H(\omega)=\sum_{k_1,\cdots,k_n}|a_{k_1\cdots k_n}|\prod_{j=1}^ng_{k_j}^{\iota_j}.\] We can calculate
\begin{equation}\label{exp1}\mathbb{E}(|F(\omega)|^{2q})=\sum_{(k_j^i,\ell_j^i)}\prod_{i=1}^qa_{k_1^i\cdots k_n^i}\overline{a_{\ell_1^i\cdots\ell_n^i}}\mathbb{E}\bigg(\prod_{i=1}^q\prod_{j=1}^n\eta_{k_j^i}^{\iota_j}\overline{\eta_{\ell_j^i}^{\iota_j}}\bigg),
\end{equation} where $1\leq i\leq q$ and $1\leq j\leq n$, and similarly for $H$,
\begin{equation}\label{exp2}\mathbb{E}(|H(\omega)|^{2q})=\sum_{(k_j^i,\ell_j^i)}\prod_{i=1}^q|a_{k_1^i\cdots k_n^i}||a_{\ell_1^i\cdots\ell_n^i}|\mathbb{E}\bigg(\prod_{i=1}^q\prod_{j=1}^ng_{k_j^i}^{\iota_j}\overline{g_{\ell_j^i}^{\iota_j}}\bigg).
\end{equation}The point is that we always have
\[\bigg|\mathbb{E}\bigg(\prod_{i=1}^q\prod_{j=1}^r\eta_{k_j^i}^{\iota_j}\overline{\eta_{\ell_j^i}^{\iota_j}}\bigg)\bigg|\leq \mathrm{Re}\mathbb{E}\bigg(\prod_{i=1}^q\prod_{j=1}^rg_{k_j^i}^{\iota_j}\overline{g_{\ell_j^i}^{\iota_j}}\bigg).\] In fact, in order for either side to be nonzero, for any particular $k$, the number of $(i,j)$'s such that $k_j^i=k$ must equal the number of $(i,j)$'s such that $\ell_j^i=k$; say both equal to $m$, then by independence, the factor that the $\eta_k^\pm$'s contribute to the expectation on the left hand side will be $\mathbb{E}|\eta_k|^{2m}=1$, while for the right hand side this factor will be $\mathbb{E}|g_k|^{2m}=m!\geq 1$.

This implies that $\mathbb{E}(|F|^{2q})\leq \mathbb{E}(|H|^{2q})$ for any positive integer $q$; since (\ref{largedevest}) holds for $H$ we have
\[\mathbb{E}(|H|^{2q})\leq (Cq)^{nq}M^q\] with an absolute constant $C$. This gives an upper bound for $\mathbb{E}(|F|^{2q})$, and by Chebyshev inequality, we deduce (\ref{largedevest}) for $F$.\end{proof}
\begin{lem}\label{counting} Let $\beta=(\beta_1,\cdots,\beta_d)\in[1,2]^d$ and $0<T\leq L^d$. Assume that $\beta$ is generic for $T\geq L^{2}$. Then, uniformly in $(k,a,b,c)\in(\mathbb{Z}_L^d)^4$ and $m\in\mathbb{R}$, the sets
\begin{multline}\label{3vectorest}
S_3=\big\{(x,y,z)\in (\mathbb{Z}_L^d)^3:x-y+z=k,\,\,\big||x|_\beta^2-|y|_\beta^2+|z|_\beta^2-|k|_\beta^2-m\big|\leq T^{-1},\\|x-a|\leq L^\theta,\,\,|y-b|\leq L^\theta,\,\,|z-c|\leq L^\theta,\,\,\mathrm{and\ }k\not\in\{x,z\}\big\}
\end{multline}
\begin{multline}\label{2vectorest}S_2=\big\{(x,y)\in(\mathbb{Z}_L^d)^3:x\pm y=k,\,\,\big||x|_\beta^2\pm|y|_\beta^2-|k|_\beta^2-m\big|\leq T^{-1},\\|x-a|\leq L^\theta,\,\,|y-b|\leq L^\theta,\,\,\mathrm{and\ }x\neq y\mathrm{\ if\ the\ sign}\pm\mathrm{is}-\big\}
\end{multline} satisfy the bounds
\begin{equation}\label{vectorbds}\#S_3\lesssim L^{2d+\theta}T^{-1},\quad \#S_2\lesssim\left\{
\begin{split}& L^{d+\theta},&\mathrm{if\ }&T\leq L;\\
& L^{d+1+\theta}T^{-1},&\mathrm{if\ }&T\in[L,L^2];\\
& L^{d-1+\theta},&\mathrm{if\ }&T\geq L^{2}\mathrm{\ and\ }\beta_i\mathrm{\ generic}.
\end{split}
\right.
\end{equation} In particular, with $\rho$ defined in (\ref{defrho}) we have
\begin{equation}\label{defq}\max\big((\#S_3)^{\frac{1}{2}},\#S_2\big)\leq L^\theta Q,\quad Q:=\frac{L^d\rho}{\alpha T}.
\end{equation}
\end{lem}
\begin{proof} We first prove the bound on $\#S_3$. Let $k-x=p$ and $k-z=q$, then we may write $p=(L^{-1}u_1,\cdots, L^{-1}u_d)$ and similarly for $q$, where each $u_i$ and $v_i$ is an integer and belongs to a fixed interval of length $O(L^{1+\theta})$. Moreover from $(x,y,z)\in S_3$ we deduce that
\[\bigg|\sum_{i=1}^d\beta_iu_iv_i+\frac{L^2m}{2}\bigg|\leq\frac{L^2T^{-1}}{2}.\] We may assume $u_iv_i=0$ for $1\leq i\leq r$, and $\sigma_i:=u_iv_i\neq 0$ for $r+1\leq i\leq d$, then the number of choices for $(u_i,v_i:1\leq i\leq r)$ is $O(L^{r+\theta})$. It is known {(see \cite{DNY,DNY2})} that given $\sigma\neq 0$, the number of integer pairs $(u,v)$ such that $u$ and $v$ each belongs to an interval of length $O(L^{1+\theta})$ and $uv=\sigma$ is $O(L^\theta)$. Therefore $\#S_3$ is bounded by $O(L^{r+\theta})$ times the number of choices for $(\sigma_{r+1},\cdots,\sigma_d)$ that satisfies
\begin{equation}\label{neweqn1}|\sigma_j|\leq L^{2+\theta}\,(r+1\leq i\leq d),\quad \sum_{j=r+1}^d\beta_i\sigma_i=\frac{L^2m}{2}+O(L^2T^{-1}).\end{equation} Using also the assumption $T\leq L^{d}$, it suffices to show that the number choices for $(\sigma_{r+1},\cdots,\sigma_d)$ is at most $O(1+L^{2(d-r)+\theta}T^{-1})$. This latter bound is trivial if $d-r=1$ or $L^2T^{-1}\geq 1$, so we may assume $d-r\geq 2$, $T\geq L^{2}$ and $\beta_i$ is generic. it is well-known in Diophantine approximation (see for example \cite{Cass} Chapter VII) that for generic $\beta_i$ we have
\[\bigg|\sum_{i=r+1}^d\beta_i\eta_i\bigg|\gtrsim\big(\max_{r+1\leq i\leq d}\langle\eta_i\rangle\big)^{-(d-r-1)-\theta}\quad\mathrm{if\ }\eta_i\mathrm{\ are\ not\ all\ }0,\] so the distance between any two points $(\sigma_i:r+1\leq i\leq d)$ and $(\sigma_i':r+1\leq i\leq d)$ satisfying (\ref{neweqn1}) is at least $(L^2T^{-1})^{-\frac{1}{d-r-1}-\theta}$. Since all these points belong to a box which has size $O(1)$ in one direction and size $O(L^{2+\theta})$ in other orthogonal directions, we deduce that the number of solutions to (\ref{neweqn1}) is at most $1+L^\theta L^{2(d-r-1)}L^2T^{-1}$, as desired.

The proof for the bound on $\#S_2$ is easier. In fact, if $T\leq L$ we trivially have $\#S_2\leq L^{d+\theta}$ as $y$ will be fixed once $x$ is; if $T\geq L$, then we may assume $x_d-y_d\neq 0$ if the sign $\pm$ is $-$, and then fix the first coordinates $x_j(1\leq j\leq d-1)$ and hence $y_j(1\leq j\leq d-1)$. Then we have that $x_d\pm y_d$ is fixed, and $x_d^2\pm y_d^2$ belongs to a fixed interval of length $O(T^{-1})$. Since $x_d,y_d\in L^{-1}\mathbb{Z}$, we know that $x_d$ has at most $1+L^2T^{-1}$ choices, which implies what we want to prove.
\end{proof}
  \subsection{Bounds for $\mathcal{J}_n$} In this section we prove Proposition \ref{lwp1}. We will need to extend the notion of ternary trees to \emph{paired, colored} ternary trees, which we define below.
  \begin{df}[Tree pairings and colorings]\label{dfpairing} Let $\Tm$ be a ternary tree as in Definition \ref{tree1}. We will pair some of the leaves of $\Tm$, such that each leaf belongs to at most one pair. The two leaves in a pair are called \emph{partners} of each other, and the unpaired leaves are called \emph{single}. We assume $\iota_{\mathfrak{l}}+\iota_{\mathfrak{l}'}=0$ for any pair $(\mathfrak{l},\mathfrak{l}')$. The set of single leaves is denoted $\mathcal{S}$. The number of pairs is denoted by $p$, so that $|\mathcal S|=l-2p$. Moreover we assume that some nodes in $\mathcal{S}\cup\{\mathfrak{r}\}$ are colored red, and let $\Rm$ be the set of red nodes. We shall denote $r=|\mathcal R|$.
 \end{df}
  \begin{figure}[h!]
  \includegraphics[width=0.58\linewidth]{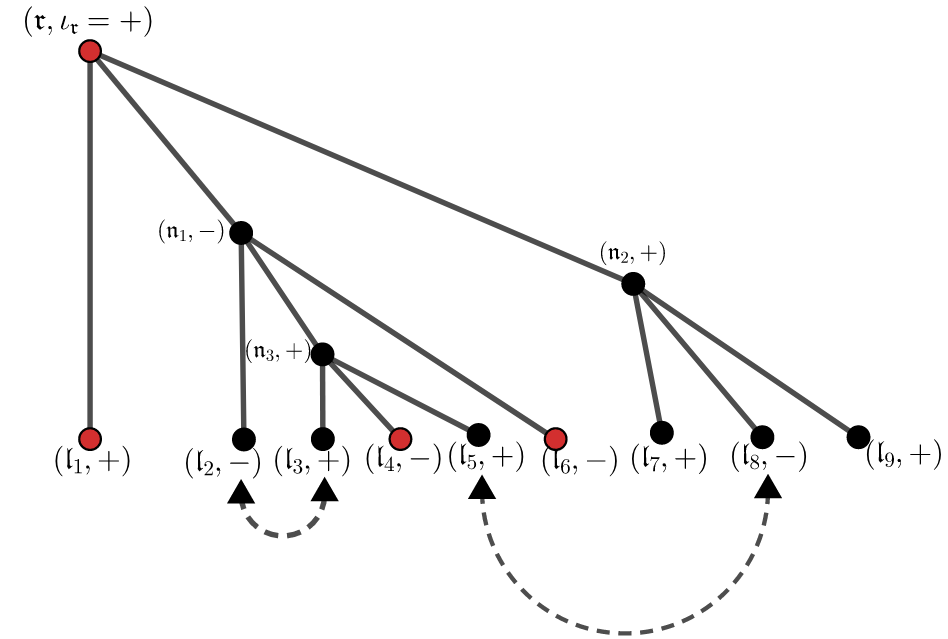}
  \caption{A paired tree with two pairings ($p=2$). The set $\mathcal S$ of single leaves is $\{\mathfrak l_1,\mathfrak l_4,\mathfrak l_6,\mathfrak l_7,\mathfrak l_9 \}$. The subset $\mathcal R\subset \mathcal{S}\cup\{\mathfrak{r}\}$ of red-colored vertices is $\{\mathfrak r, \mathfrak l_1,\mathfrak l_4,\mathfrak l_6\}$. Here $(l, p, r)=(9, 2, 4)$. A strongly admissible assignment with respect to the above pairing, coloring, and to a certain fixed choice of the red modes $(k_{\mathfrak r},k_{\mathfrak l_4},k_{\mathfrak l_6})$, corresponds to having the modes $k_{\mathfrak l_2}=k_{\mathfrak l_3}$, $k_{\mathfrak l_5}=k_{\mathfrak l_8}$, and $|k_{\mathfrak l}|\leq L^\theta$ for all the uncolored leaves. The rest of the modes are determined according to Definition \ref{tree1}.}
  \label{fig:Trees2}
\end{figure}
We shall use the red coloring above to denote that the frequency assignments to the corresponding red vertex is fixed in the counting process. We also introduce the following definition. 
\begin{df}[Strong admissibility]
Suppose we fix $n_{\mathfrak{m}}\in \mathbb{Z}_L^d$ for each $\mathfrak{m}\in\Rm$. An assignment $(k_{\mathfrak{n}}:\mathfrak{n}\in\Tm)$ is called \emph{strongly admissible} with respect to the given pairing, coloring, and $(n_{\mathfrak{m}}:\mathfrak{m}\in\Rm)$, if it is admissible in the sense of Definition \ref{tree1}, and 
\begin{equation}\label{stradm}k_{\mathfrak{m}}=n_{\mathfrak{m}},\,\forall \mathfrak{m}\in\Rm;\quad |k_{\mathfrak{l}}|\leq L^\theta,\,\forall \mathfrak{l}\in\Lm;\quad k_{\mathfrak{l}}=k_{\mathfrak{l}'},\,\forall\text{ pair of leaves $(\mathfrak{l},\mathfrak{l}')$}.
\end{equation}
\end{df}
\medskip

 The key to the proof of Proposition \ref{lwp1} is the following combinatorial counting bound.
\begin{prop}\label{countingbd} Let $\Tm$ be a paired and colored ternary tree such that $\Rm\neq\varnothing$, and let $(n_{\mathfrak{m}}:\mathfrak{m}\in\Rm)$ be fixed. We also fix $\sigma_{\mathfrak{n}}\in\mathbb{R}$ for each $\mathfrak{n}\in\Nm$. Let $l=|\Lm|$ be the total number of leaves, $p$ be the number of pairs, and $r=|\Rm|$ be the number of red nodes. Then, the number of strongly admissible assignments $(k_{\mathfrak{n}}:\mathfrak{n}\in\Tm)$, which also satisfies
\begin{equation}\label{admissibility2}|\Omega_{\mathfrak{n}}-\sigma_{\mathfrak{n}}|\leq T^{-1},\,\forall \mathfrak{n}\in\Nm,
\end{equation} is bounded by (recall $Q$ defined in (\ref{defq}))
\begin{equation}\label{countingbd0}M\leq\left\{\begin{aligned}
&L^\theta Q^{l-p-r},&\mathrm{\ if\ }&\Rm\neq\mathcal{S}\cup\{\mathfrak{r}\},\\
& L^\theta Q^{l-p-r+1},&\mathrm{\ if\ }&\Rm=\mathcal{S}\cup\{\mathfrak{r}\}.
\end{aligned}
\right.
\end{equation}
\end{prop}
\begin{proof} We proceed by induction. The base cases directly follow from Lemma \ref{counting}. Now suppose the desired bound holds for all smaller trees, and consider $\Tm$. Let $\mathfrak{r}_1,\mathfrak{r}_2,\mathfrak{r}_3$ be the children of the root node $\mathfrak{r}$, and $\Tm_j$ be the subtree rooted at $\mathfrak{r}_j$. Let $l_j$ be the number of leaves in $\Tm_j$, $p_j$ be the number of pairs within $\Tm_j$, $p_{ij}$ be the number of pairings between $\Tm_i$ and $\Tm_j$ and $r_j=|\mathcal{R}\cap\Tm_j|$, then we have
\[l=l_1+l_2+l_3,\quad p=p_1+p_2+p_3+p_{12}+p_{13}+p_{23},\quad r=r_1+r_2+r_3+\mathbf{1}_{\mathfrak{r}\in\Rm}.\] Also note that $|k_{\mathfrak{n}}|\lesssim L^\theta$ for all $\mathfrak{n}\in\Tm$.

The proof below will be completely algorithmic with the discussion of a lot of cases. The general strategy is to perform the following four operations in a suitable order: (a) apply Lemma \ref{counting} to count the number of choices for the values among $\{k_{\mathfrak{r}},k_{\mathfrak{r}_1},k_{\mathfrak{r}_2},k_{\mathfrak{r}_2}\}$ that is not already fixed (this step may be trivial if three of these four vectors are already fixed (colored), or if one of them is already fixed, and $k_{\mathfrak{r}}=k_{\mathfrak{r}_1}=k_{\mathfrak{r}_2}=k_{\mathfrak{r}_3}$); (b) apply the induction hypothesis to one of the subtrees $\Tm_j$ and count the number of choices for $(k_{\mathfrak{n}}:\mathfrak{n}\in\Tm_j)$. Let these operations be denoted by $\mathcal{O}_j(0\leq j\leq 3)$, and the associated number of choices be $M_j$, with superscripts indicating different cases. In the whole process we may color new nodes $\mathfrak{n}$ red if $k_{\mathfrak{n}}$ is already fixed during the previous operations, namely when $\mathfrak{n}=\mathfrak{r}$ and we have performed $\mathcal{O}_0$ before, or when $\mathfrak{n}=\mathfrak{r}_j$ and we have performed $\mathcal{O}_0$ or $\mathcal{O}_j$ before, or when $\mathfrak{n}$ is a leaf that has a partner in $\Tm_j$ and we have performed $\mathcal{O}_j$ before.

(1) Suppose $\mathfrak{r}\not\in\Rm$, then we may assume that there is a red leaf from\footnote{Strictly speaking the roles of $\Tm_1$ and $\Tm_2$ are not exactly symmetric due to the sign difference, but this will not affect the proof below since Lemma \ref{counting} includes all choices of signs.} $\Tm_1$. We first perform $\mathcal{O}_1$ and get a factor
\[M_1^{(1)}:= L^\theta Q^{l_1-p_1-r_1}.\]Now $\mathfrak{r}_1$ is colored red, as is any leaf in $\Tm_2\cup\Tm_3$ which has a partner in $\Tm_1$. There are then two cases.

(1.1) Suppose now there is a leaf in $\Tm_2\cup\Tm_3$, say from $\Tm_2$, is red. Then we will perform $\mathcal{O}_2$ and get a factor
\[M_2^{(1.1)}:=L^\theta Q^{l_2-p_2-r_2-p_{12}}.\] Now $\mathfrak{r}_2$ is colored red, as is any leaf of $\Tm_3$ which has a partner in $\Tm_2$. There are again two cases:

(1.1.1) Suppose now there is a red leaf in $\Tm_3$, then we will perform $\mathcal{O}_3$ and get a factor
\[M_3^{(1.1.1)}:=L^\theta Q^{l_3-p_3-r_3-p_{13}-p_{23}},\] then color $\mathfrak{r}_3$ red and apply $\mathcal{O}_0$ and get a factor $M_0^{(1.1.1)}:=1$. Thus
\[M\leq M_1^{(1)}M_2^{(1.1)}M_3^{(1.1.1)}M_0^{(1.1.1)}=L^{l-p-r+\theta},\] which is what we need.

(1.1.2) Suppose after the step in (1.1), there is no red leaf in $\Tm_3$, then $r_3=p_{13}=p_{23}=0$. We will perform $\mathcal{O}_0$ and get a factor $M_0^{(1.1.2)}:=L^\theta Q^{1}$ (perhaps with slightly enlarged $\theta$; same below). Now we may color $\mathfrak{r}_3$ red and perform $\mathcal{O}_3$ to get a factor
\[M_3^{(1.1.2)}:=L^\theta Q^{l_3-p_3-1}.\] Thus
\[M\leq M_1^{(1)}M_2^{(1.1)}M_0^{(1.1.2)}M_3^{(1.1.2)}=L^\theta Q^{l-p-r},\] which is what we need.

(1.2) Now suppose that after the step in (1), there is no red leaf in $\Tm_2\cup\Tm_3$, then $r_2=r_3=p_{12}=p_{13}=0$. There are two cases:

(1.2.1) Suppose there is a single leaf in $\Tm_2\cup\Tm_3$, say from $\Tm_2$. Then we will perform $\mathcal{O}_0$ and get a factor $M_0^{(1.2.1)}:=L^\theta Q^{2}$. Now we may color $\mathfrak{r}_2$ and $\mathfrak{r}_3$ red and perform $\mathcal{O}_3$ to get a factor
\[M_3^{(1.2.1)}:=L^\theta Q^{l_3-p_3-1}.\] Now any leaf of $\Tm_2$ which has a partner in $\Tm_3$ is colored red, so we may perform $\mathcal{O}_2$ and get a factor
\[M_2^{(1.2.1)}:=L^\theta Q^{l_2-p_2-p_{23}-1}.\] Thus 
\[M\leq M_1^{(1)}M_0^{(1.2.1)}M_3^{(1.2.1)}M_2^{(1.2.1)}=L^\theta Q^{l-p-r},\] which is what we need.

(1.2.2) Suppose there is no single leaf in $\Tm_2\cup\Tm_3$, then all leaves in $\Tm_2\cup\Tm_3$ are paired to one another, which implies $k_{\mathfrak{r}_2}=k_{\mathfrak{r}_3}$ and $\mathfrak{r}_2$ and $\mathfrak{r}_3$ have opposite signs, hence by admissibility condition we must have $k_{\mathfrak{r}}=k_{\mathfrak{r}_1}=k_{\mathfrak{r}_2}=k_{\mathfrak{r}_3}$. This allows us to perform $\mathcal{O}_0$ and color $\mathfrak{r}_2$ and $\mathfrak{r}_3$ red with $M_0^{(1.2.2)}:=1$, then perform $\mathcal{O}_3$ and color red any leaf of $\Tm_2$ which has a partner in $\Tm_3$, and then perform $\mathcal{O}_2$ (for which we use the second bound in \eqref{countingbd0}). This leads to the factors
\[M_3^{(1.2.2)}:=L^\theta Q^{l_3-p_3-1},\quad M_2^{(1.2.2)}\leq L^\theta Q^{l_2-p_2-p_{23}-1+1},\] thus
\[M\leq M_1^{(1)}M_0^{(1,2,2)}M_3^{(1.2.2)}M_2^{(1.2.2)}=L^\theta Q^{l-p-r-1},\] which is better than what we need.

(2) Now suppose $\mathfrak{r}\in\Rm$, then $r=r_1+r_2+r_3+1$. There are two cases.

(2.1) Suppose there is one single leaf that is not red, say from $\Tm_1$. There are again two cases.

(2.1.1) Suppose there is a red leaf in $\Tm_2\cup\Tm_3$, say $\Tm_2$. Then we perform $\mathcal{O}_2$ and get a factor
\[M_2^{(2.1.1)}:=L^\theta Q^{l_2-p_2-r_2}.\] We now color red $\mathfrak{r}_2$ and any leaf in $\Tm_1\cup\Tm_3$ which has a partner in $\Tm_2$. There are further two cases:

(2.1.1.1) Suppose now there is a red leaf in $\Tm_3$, then we perform $\mathcal{O}_3$ and get a factor
\[M_3^{(2.1.1.1)}:= L^\theta Q^{l_3-p_3-r_3-p_{23}}.\] Now we perform $\mathcal{O}_0$ and get a factor $M_0^{(2.1.1.1)}:=1$, then color red $\mathfrak{r}_1$ as well as any leaf of $\Tm_1$ which has a partner in $\Tm_3$, and perform $\mathcal{O}_1$ to get a factor
\[M_1^{(2.1.1.1)}:=L^\theta Q^{l_1-p_1-r_1-p_{12}-p_{13}-1}.\] Thus
\[M\leq M_2^{(2.1.1)}M_3^{(2.1.1.1)}M_0^{(2.1.1.1)}M_{1}^{(2.1.1.1)}=L^\theta Q^{l-p-r},\] which is what we need.

(2.1.1.2) Suppose after the step in (2.1.1) there is no red leaf in $\Tm_3$, then $r_3=p_{23}=0$. We perform $\mathcal{O}_0$ and get a factor $M_0^{(2.1.1.2)}:=L^\theta Q^{1}$. Then we color $\mathfrak{r}_1$ and $\mathfrak{r}_3$ red and perform $\mathcal{O}_3$ to get a factor
\[M_3^{(2.1.1.2)}:=L^\theta Q^{l_3-p_3-1}.\] Finally we color red any leaf of $\Tm_1$ which has a partner in $\Tm_3$, and perform $\mathcal{O}_1$ to get a factor
\[M_1^{(2.1.1.2)}:=L^\theta Q^{l_1-p_1-r_1-p_{12}-p_{13}-1}.\]
Thus
\[M\leq M_2^{(2.1.1)}M_0^{(2.1.1.2)}M_3^{(2.1.1.2)}M_{1}^{(2.1.1.2)}=L^\theta Q^{l-p-r},\] which is what we need.

(2.1.2) Suppose in the beginning there is no red leaf in $\Tm_2\cup\Tm_3$, then $r_2=r_3=0$. There are again two cases.

(2.1.2.1) Suppose there is a leaf in $\Tm_2\cup\Tm_3$, say from $\Tm_2$, that is either single or paired with a leaf in $\mathcal{T}_1$. Then we will perform $\mathcal{O}_0$ and get a factor $M_0^{(2.1.2.1)}:=L^\theta Q^{2}$. After this we will color $\mathfrak{r}_1,\mathfrak{r}_2,\mathfrak{r}_3$ red and perform $\mathcal{O}_3$ to get a factor
\[M_3^{(2.1.2.1)}:= L^\theta Q^{l_3-p_3-1}.\] We then color red any leaf of $\Tm_1$ and $\Tm_2$ which has a partner in $\Tm_3$ and perform $\mathcal{O}_2$ to get a factor
\[M_2^{(2.1.2.1)}:= L^\theta Q^{l_2-p_2-p_{23}-1}.\] Finally we color red any leaf of $\Tm_1$ which has a partner in $\Tm_2$, and perform $\mathcal{O}_1$ to get a factor
\[M_{1}^{(2.1.2.1)}:= L^\theta Q^{l_1-p_1-r_1-p_{12}-p_{13}-1}.\] Thus
\[M\leq M_0^{(2.1.2.1)}M_3^{(2.1.2.1)}M_2^{(2.1.2.1)}M_{1}^{(2.1.2.1)}=L^\theta Q^{l-p-r},\] which is what we need.

(2.1.2.2) Suppose there is no leaf in $\Tm_2\cup\Tm_3$ that is either single or paired with a leaf in $\mathcal{T}_1$, then in the same way as (1.2.2) we must have $k_{\mathfrak{r}}=k_{\mathfrak{r}_1}=k_{\mathfrak{r}_2}=k_{\mathfrak{r}_3}$. Moreover we have $p_{12}=p_{13}=0$. Then we will perform $\mathcal{O}_0$ and get a factor $M_0^{(2.1.2.2)}:=1$. After this we will color $\mathfrak{r}_1,\mathfrak{r}_2,\mathfrak{r}_3$ red and perform $\mathcal{O}_3$ to get a factor \[M_3^{(2.1.2.2)}:= L^\theta Q^{l_3-p_3-1}.\] We then color red any leaf of $\Tm_2$ which has a partner in $\Tm_3$ and perform $\mathcal{O}_2$ to get a factor
\[M_2^{(2.1.2.2)}\leq L^\theta Q^{l_2-p_2-p_{23}-1+1}.\] Finally we perform $\mathcal{O}_1$, again using the second part of estimate \eqref{countingbd0}, to get a factor
\[M_{1}^{(2.1.2.2)}:= L^\theta Q^{l_1-p_1-r_1-1}.\]
 Thus
\[M\leq M_0^{(2.1.2.2)}M_3^{(2.1.2.2)}M_2^{(2.1.2.2)}M_{1}^{(2.1.2.2)}=L^\theta Q^{l-p-r-1},\] which is better than what we need.

(2.2) Now suppose that in the beginning all single leaves are red, i.e. $\Rm=\mathcal{S}\cup\{\mathfrak{r}\}$. Then we can argue in exactly the same way as in (2.1), except that in the last step where we perform $\mathcal{O}_1$, it may happen that the root $\mathfrak{r}_1$ as well as all leaves of $\Tm_1$ are red at that time, so we lose one power of $Q$ in view of the weaker bound from the induction hypothesis. However since $\Rm=\mathcal{S}\cup\{\mathfrak{r}\}$, we are in fact allowed to lose this power, so we can still close the inductive step, in the same way as (2.1). This completes the proof.
\end{proof}
\begin{cor}\label{imprcor} In Proposition \ref{countingbd}, suppose $\mathcal{R}=\{\mathfrak{r}\}$. Then (\ref{countingbd0}) can be improved to
\begin{equation}\label{improved}M\leq L^\theta Q^{l-p-3}L^{2d}T^{-1}.
\end{equation}
\end{cor}
\begin{proof} In the proof of Proposition \ref{countingbd}, we are now in case (2.1.2). In each sub-case, either (2.1.2.1) or (2.1.2.2) we are performing the operation $\mathcal{O}_0$ first. In case (2.1.2.1) by Lemma \ref{counting} we can replace the bound $M_0^{(2.1.2.1)}$ by $M_0':= L^\theta L^{2d}T^{-1}$, so we get
\[M\leq M_0'M_3^{(2.1.2.1)}M_2^{(2.1.2.1)}M_{1}^{(2.1.2.1)}=L^\theta Q^{l-p-3}L^{2d}T^{-1}.\] In case (2.1.2.2) we get an improvement, namely in this case we have $M\leq L^\theta Q^{l-p-2}$, which also implies (\ref{improved}), since we can check $Q\leq L^{2d}T^{-1}\leq Q^2$ by definition.
\end{proof}
Now we are ready to prove Proposition \ref{lwp1}.
\begin{proof}[Proof of Proposition \ref{lwp1}] We start with the formula (\ref{formulajt}). Let $|\Tm|=3n+1$. Due to the rapid decay of $\sqrt{n_{\mathrm{in}}}$, we may assume in the summation that $|k_{\mathfrak{l}}|\leq L^\theta$ for any $\mathfrak{l}\in\Lm$, and so $|k|\leq L^\theta$ also. For any \emph{fixed} value of $\tau$, we may apply Lemma \ref{largedev} to $L$-certainly estimate $(\widetilde{\Jm_{\Tm}})_k(\tau)$. Namely, $L$-certainly we have, for some choice of pairing and with coloring $\Rm=\{\mathfrak{r}\}$ and $n_{\mathfrak{r}}=k$, that
\begin{equation}\label{pointwisebd}\langle k\rangle^{4s}|(\widetilde{\Jm_{\Tm}})_k(\tau)|^2\leq L^\theta\bigg(\frac{\alpha T}{L^{d}}\bigg)^{2n}\sum_{(k_{\mathfrak{l}}:\mathfrak{l}\in\mathcal{S})}\bigg[\sum_{(k_{\mathfrak{l}}:\mathfrak{l}\in\Lm\backslash\mathcal{S})}^{**}\big|\mathcal{K}_{\Tm}(\tau,k_{\mathfrak{n}}:\mathfrak{n}\in\Tm)\big|\bigg]^2,
\end{equation} where $\sum^{**}$ represents summation under the condition that the unique admissible assignment determined by $(k_{\mathfrak{l}}:\mathfrak{l}\in\mathcal{S})$ and $(k_{\mathfrak{l}}:\mathfrak{l}\in\Lm\backslash\mathcal{S})$ is strongly admissible. Next we would like to assume (\ref{pointwisebd}) for \emph{all} $\tau$, which can be done by the following trick. First due to the decay factor in (\ref{boundk}) and the assumption $|k_{\mathfrak{l}}|\leq L^\theta$, we may assume $|\tau|\leq L^{d+\theta}$; moreover, choosing a large power $D$, we may divide the interval $[-L^\theta,L^\theta]$ into subintervals of length $L^{-D}$ and pick one point $\tau_j$ from each interval. Due to the differentiability of $\mathcal{K}_{\Tm}$, see (\ref{boundk}), we can bound the difference \[\big|\mathcal{K}_{\Tm}(\tau,k_{\mathfrak{n}}:\mathfrak{n}\in\Tm)-\mathcal{K}_{\Tm}(\tau_j,k_{\mathfrak{n}}:\mathfrak{n}\in\Tm)\big|\] by a large negative power of $L$ provided $\tau$ is in the same interval as $\tau_j$. Therefore, as long as (\ref{pointwisebd}) is true for each $\tau_j$ we can assume it is true for each $\tau$ up to negligible errors. Since the number of $\tau_j$'s is at most $O(L^{2D})$ and (\ref{pointwisebd}) holds $L$-certainly for each fixed $\tau_j$, we conclude that $L$-certainly, (\ref{pointwisebd}) holds for all $\tau$.

Now, by expanding the square in (\ref{pointwisebd}), it suffices to bound the quantity
\[\int_{\mathbb{R}}\langle\tau\rangle^{2b}\sum_{(k_{\mathfrak{l}}:\mathfrak{l}\in\mathcal{S})}\sum_{(k_{\mathfrak{l}}:\mathfrak{l}\in\Lm\backslash\mathcal{S})}^{**}\sum_{(k_{\mathfrak{l}}':\mathfrak{l}\in\Lm\backslash\mathcal{S})}^{**'}\big|\mathcal{K}_{\Tm}(\tau,k_{\mathfrak{n}}:\mathfrak{n}\in\Tm)\big|\cdot\big|\mathcal{K}_{\Tm}(\tau,k_{\mathfrak{n}}':\mathfrak{n}\in\Tm)\big|\,\mathrm{d}\tau,\] where $(k_{\mathfrak{n}}:\mathfrak{n}\in\Tm)$ is the unique admissible assignment determined by $(k_{\mathfrak{l}}:\mathfrak{l}\in\Lm)$ and $(k_{\mathfrak{l}}:\mathfrak{l}\in\Lm\backslash\mathcal{S})$, and $(k_{\mathfrak{n}}':\mathfrak{n}\in\Tm)$ is the one determined by $(k_{\mathfrak{l}}:\mathfrak{l}\in\Lm)$ and $(k_{\mathfrak{l}}':\mathfrak{l}\in\Lm\backslash\mathcal{S})$. The conditions in the summations $\sum^{**}$ and $\sum^{**'}$ correspond to these two assignments being strongly admissible. By (\ref{boundk}) we have (for some choice of $d_{\mathfrak{n}}$)
\[\langle\tau\rangle^{2b}\big|\mathcal{K}_{\Tm}(\tau,k_{\mathfrak{n}}:\mathfrak{n}\in\Tm)\big|\cdot\big|\mathcal{K}_{\Tm}(\tau,k_{\mathfrak{n}}':\mathfrak{n}\in\Tm)\big|\lesssim\langle\tau\rangle^{2b}\langle \tau-Td_{\mathfrak{r}}q_{\mathfrak{r}}\rangle^{-10}\langle \tau-Td_{\mathfrak{r}}q_{\mathfrak{r}}'\rangle^{-10}\prod_{\mathfrak{n}\in\Nm}\langle Tq_{\mathfrak{n}}\rangle^{-1}\langle Tq_{\mathfrak{n}}'\rangle^{-1},\] where $q_{\mathfrak{n}}$ and $q_{\mathfrak{n}}'$ are defined from the assignment $(k_{\mathfrak{n}})$ and $(k_{\mathfrak{n}}')$ respectively via (\ref{defqn}). Thus the integral in $\tau$ gives
\[\max(\langle Tq_{\mathfrak{r}}\rangle,\langle Tq_{\mathfrak{r}}'\rangle)^{-2+2b}\langle T(q_{\mathfrak{r}}-q_{\mathfrak{r}}')\rangle^{-5}\prod_{\mathfrak{r}\neq\mathfrak{n}\in\Nm}\langle Tq_{\mathfrak{n}}\rangle^{-1}\langle Tq_{\mathfrak{n}}'\rangle^{-1},\] and it suffices to bound 
\[\sum_{(k_{\mathfrak{l}}:\mathfrak{l}\in\mathcal{S})}\sum_{(k_{\mathfrak{l}}:\mathfrak{l}\in\Lm\backslash\mathcal{S})}^{**}\sum_{(k_{\mathfrak{l}}':\mathfrak{l}\in\Lm\backslash\mathcal{S})}^{**'}\max(\langle Tq_{\mathfrak{r}}\rangle,\langle Tq_{\mathfrak{r}}'\rangle)^{-2+2b}\langle T(q_{\mathfrak{r}}-q_{\mathfrak{r}}')\rangle^{-5}\prod_{\mathfrak{r}\neq\mathfrak{n}\in\Nm}\langle Tq_{\mathfrak{n}}\rangle^{-1}\langle Tq_{\mathfrak{n}}'\rangle^{-1}.\]Since all the $q$'s are bounded by $L^\theta$, and $T\leq L^d$, we may fix the integer parts of each $Tq_{\mathfrak{n}}$ and $Tq_{\mathfrak{n}}'$ for each $\mathfrak{n}\in\Nm$, and reduce the above sum to a counting bound, at the price of losing a power $L^{C(b-\frac{1}{2})}$. Now by definition (\ref{defqn}), each $q_{\mathfrak{n}}$ is a linear combination of $\Omega_{\mathfrak{n}}$'s and conversely each $\Omega_{\mathfrak{n}}$ is a linear combination of $q_{\mathfrak{n}}$'s. So once the integer parts of each $Tq_{\mathfrak{n}}$ and $Tq_{\mathfrak{n}}'$ is fixed, we have also fixed $\sigma_{\mathfrak{n}}\in\mathbb{R}$ and $\sigma_{\mathfrak{n}}'\in\mathbb{R}$ such that \begin{equation}\label{extra}|\Omega_{\mathfrak{n}}-\sigma_{\mathfrak{n}}|\leq T^{-1},\quad |\Omega_{\mathfrak{n}}'-\sigma_{\mathfrak{n}}'|\leq T^{-1}.\end{equation}

Therefore we are reduced to counting the number of $(k_{\mathfrak{l}}:\mathfrak{l}\in\mathcal{S})$, $(k_{\mathfrak{l}}:\mathfrak{l}\in\Lm\backslash\mathcal{S})$ and $(k_{\mathfrak{l}}':\mathfrak{l}\in\Lm\backslash\mathcal{S})$ such that the assignments $(k_{\mathfrak{n}})$ and $(k_{\mathfrak{n}}')$ are both strongly admissible and satisfy (\ref{extra}). Now let $|\Lm|=l=2n+1$ and $p$ be the number of pairs, then $|\mathcal{S}|=2n+1-2p$. First we count the number of choices for $(k_{\mathfrak{l}}:\mathfrak{l}\in\mathcal{S})$ and $(k_{\mathfrak{l}}:\mathfrak{l}\in\Lm\backslash\mathcal{S})$, where we apply Corollary \ref{imprcor} with $\Rm=\{\mathfrak{r}\}$ and get the factor $M:=L^\theta Q^{2n-p-2}L^{2d}T^{-1}$; then, with $k_{\mathfrak{l}}$ fixed for all $\mathfrak{l}\in\mathcal{S}$, we will count the number of choices for $(k_{\mathfrak{l}}':\mathfrak{l}\in\Lm\backslash\mathcal{S})$ by applying Proposition \ref{countingbd} with $\Rm=\mathcal{S}\cup\{\mathfrak{r}\}$ and get the factor $M':= L^{\theta}Q^p$. In the end we get that $L$-certainly,
\begin{multline*}\sup_k\langle k\rangle^{4s}|(\widetilde{\Jm_{\Tm}})_k(\tau)|^2\leq L^{\theta+C(b-\frac{1}{2})}\bigg(\frac{\alpha T}{L^{d}}\bigg)^{2n}MM'\\\leq L^{\theta+C(b-\frac{1}{2})}\bigg(\frac{\alpha T}{L^{d}}\bigg)^{2n}Q^{2n-2}L^{2d}T^{-1}=L^{\theta+C(b-\frac{1}{2})}\rho^{2n-2}(\alpha^2T)\end{multline*} by the definition of $Q$ in (\ref{defq}), as desired.
\end{proof}
  \subsection{Bounds for $\mathcal{P}_{\pm}$}\label{operatorsec} In this section we prove Proposition \ref{lwp3}. The proof for $\mathcal{P}_{\pm}$ are similar, so we will only consider $\mathcal{P}_+$. 
  
\begin{proof}[Proof of Proposition \ref{lwp3}]
\underline{Step 1: First reductions.} We start with some simple observations. The operator $\mathcal{P}_+(v)=\mathcal{IW}(\Jm_{\Tm_1},\Jm_{\Tm_2},v)$, where $\mathcal{I}$ and $\mathcal{W}$ are defined in (\ref{Duhamel00}) and (\ref{Duhamel0}). Now in (\ref{Duhamel0}) we may assume $|k_1|,|k_2|\leq L^\theta$ due to the same reason as in the proof of Proposition \ref{lwp1}. As such, we have  \[L^{-\theta}\leq\frac{\langle k\rangle^s}{\langle k_3\rangle^s}\leq L^{\theta},\] so instead of $h^{s,b}$ bounds we only need to consider $h^{0,b}$ bounds. Next notice that, if $\mathcal{I}$ is defined by (\ref{Duhamel00}) and $\mathcal{I}_1$ is defined by $\mathcal{I}_1F=\chi\cdot(\mathrm{sgn}*(\chi\cdot F))$, then we have the identity $2\mathcal{I}F(t)=\mathcal{I}_1F(t)-\chi(t)\mathcal{I}_1F(0)$, so for $b>\frac{1}{2}$ we have $\|\mathcal{I}F\|_{h^{s,b}}\lesssim \|\mathcal{I}_1F\|_{h^{s,b}}$. Therefore, in estimating $\mathcal{P}_+$ we may replace the  operator $G$ that appears in the formula for $\mathcal{I}$ by $\mathcal{I}_1$. The advantage is that $\mathcal{I}_1$ has a formula
\[\widetilde{\mathcal{I}_1F}(\tau)=\int_{\mathbb{R}}I_1(\tau,\sigma)\widetilde{F}(\sigma)\,\mathrm{d}\sigma\] where $I_1$ is as in Lemma \ref{duhamfourier}, so we may get rid of the $I_0$ term. From now on we will stick to the renewed definition of $\mathcal{I}$. Next, by Proposition \ref{lwp1} we have the trivial bound
\begin{multline*}\|\mathcal{P}_+v\|_{h^{0,1}}\sim\|\mathcal{P}_+v\|_{\ell_k^2L_t^2}+\|\partial_t\mathcal{P}_+v\|_{\ell_k^2L_t^2}\\\lesssim\|v\|_{\ell_k^2L_t^2}\cdot\frac{\alpha T}{L^{d}}\sum_{k_1,k_2}\|(\Jm_{\Tm_1})_{k_1}\|_{L_t^\infty}\|(\Jm_{\Tm_2})_{k_2}\|_{L_t^\infty}\lesssim\alpha TL^{d+\theta+C(b-\frac{1}{2})}\rho^{n_1+n_2} \cdot\|v\|_{h^{0,b}}.\end{multline*} Note also that $\alpha T\leq L^{d}$ and $\rho\leq L^{-\varepsilon}$, so by interpolation it suffices to $L$-certainly bound the $h^{0,b}\to h^{0,1-b}$ norm of (the renewed version of) $\mathcal{P}_+$ by $L^\theta\rho^{n_1+n_2+1}$.

Now, using Lemma \ref{duhamfourier} and noticing that the bound (\ref{duhamelbound}) is symmetric in $\sigma$ and $\tau$, we have the formula 
\begin{multline}\label{formulap1}
(\widetilde{\mathcal{P}_+v})_k(\tau)=\frac{i\alpha T}{L^{d}}\langle\tau\rangle^{-1}\sum_{(m_1,m_2,k')}^*\int_{\mathbb{R}^3}J(\tau,\sigma_1-\sigma_2+\tau'+T\Omega(m_1,m_2,k',k))\\\times(\widetilde{\Jm_{\Tm_1}})_{m_1}(\sigma_1)\overline{(\widetilde{\Jm_{\Tm_2}})_{m_2}(\sigma_2)}\cdot\widetilde{v}_{k'}(\tau')\,\mathrm{d}\sigma_1\mathrm{d}\sigma_2\mathrm{d}\tau',\end{multline} where $J=J(\tau,\eta)$, as well as all its derivatives, are bounded by $\langle\tau-\eta\rangle^{-10}$. By elementary estimates we have
\begin{equation}\label{holder}\|\widetilde{w}_{k}(\tau)\|_{L_\tau^1\ell_k^2}\lesssim\|\langle \tau\rangle^b\widetilde{w}_{k}(\tau)\|_{\ell_k^2L_\tau^2},\quad\|\langle\tau\rangle^{1-b}\langle\tau\rangle^{-1}w_k(\tau)\|_{\ell_k^2L_\tau^2}\lesssim \|\widetilde{w}_{k}(\tau)\|_{L_\tau^\infty\ell_k^2},\end{equation} and thus it suffices to $L$-certainly bound the $\ell^2\to\ell^2$ norm of the operator
\begin{multline}\mathcal{X}:(\mathcal{X}v)_k=\frac{\alpha T}{L^{d}}\sum_{(m_1,m_2,k')}^*v_{k'}\cdot\int_{\mathbb{R}^2}J(\tau,\sigma_1-\sigma_2+\tau'+T\Omega(m_1,m_2,k',k))\\\times (\widetilde{\Jm_{\Tm_1}})_{m_1}(\sigma_1)\overline{(\widetilde{\Jm_{\Tm_2}})_{m_2}(\sigma_2)}\,\mathrm{d}\sigma_1\mathrm{d}\sigma_2\end{multline} \emph{uniformly} in $\tau$ and $\tau'$. 
\medskip

\underline{Step 2: Second reductions.} At this point we will apply similar arguments as in the proof of Proposition \ref{lwp1}. Namely we first restrict $|\tau|,|\tau'|\leq L^{\theta^{-1}}$ (otherwise we can gain a power of either $|\tau|^{\frac{1}{2}(b-\frac{1}{2})}$ or $|\tau'|^{\frac{1}{2}(b-\frac{1}{2})}$ from the extra room when applying (\ref{holder}) which turns into a large power of $L$ and closes the whole estimate), and then divide this interval into subintervals of length $L^{-\theta^{-1}}$ and apply differentiability to reduce to $O(L^{C\theta^{-1}})$ choices of $(\tau,\tau')$. Therefore, it suffices to \emph{fix} $\tau$ and $\tau'$ and $L$-certainly bound $\|\mathcal{X}\|_{\ell^2\to\ell^2}$. Let $\tau-\tau'=\zeta$ be fixed.

Now use (\ref{formulajt}) for the $\Jm_{\Tm_j}$ factors, assuming also $|k_{\mathfrak{l}}|\leq L^\theta$ in each tree, and integrate in $(\sigma_1,\sigma_2)$. This leads to further reduced expression for $\mathcal{X}$, which can be described as follows. First let the tree $\Tm$ be defined such that its root is $\mathfrak{r}$ and three subtrees from left to right are $\Tm_1$, $\Tm_2$ and a single node $\mathfrak{r}'$. Then we have
\[(\mathcal{X}v)_k=\sum_{k'}\mathcal{X}_{kk'}v_{k'},\] where the matrix coefficients are given by
\[\mathcal{X}_{kk'}=\bigg(\frac{\alpha T}{L^{d}}\bigg)^{n_1+n_2+1}\sum_{(k_{\mathfrak{n}}:\mathfrak{n}\in\Tm)}\mathcal{K}(|k|_\beta^2-|k'|_\beta^2,\,k_{\mathfrak{l}}:\mathfrak{r}'\neq\mathfrak{l}\in\Lm)\cdot\frac{1}{\langle Tq_{\mathfrak{r}}-\zeta\rangle^5}\prod_{\mathfrak{n}\in\Nm\setminus \{\mathfrak{r}\}}\frac{1}{\langle Tq_{\mathfrak{n}}\rangle}\prod_{\mathfrak{l}\in\Lm\setminus \{\mathfrak{r}'\}}\eta_{k_{\mathfrak{l}}}^{\iota_{\mathfrak{l}}},\]

where the sum is taken over all admissible assignments $(k_{\mathfrak{n}}:\mathfrak{n}\in\Tm)$ which satisfies $k_{\mathfrak{r}}=k$, $k_{\mathfrak{r}'}=k'$ and $|k_{\mathfrak{n}}|\leq L^{\theta}$ for $\mathfrak{n}\not\in\{\mathfrak{r},\mathfrak{r}'\}$, and the coefficient satisfies $|\mathcal{K}|\leq L^{\theta}$ and $|\partial \mathcal{K}|\leq L^\theta T$. Moreover, we observe that $\mathcal{K}$ and $q_{\mathfrak r}$ depends on the variables $k_{\mathfrak{r}}=k$ and $k_{\mathfrak{r}'}=k'$ \emph{only through the quantity} $|k|_\beta^2-|k'|_\beta^2$.

Next, we argue in the same way as in the proof of Proposition \ref{lwp1} and fix the integer parts of $Tq_{\mathfrak{n}}$ for $\mathfrak{n}\in\Nm\setminus \{\mathfrak{r}\}$, as well as the integer part of $Tq_{\mathfrak{r}}-\zeta$, at a cost of $(\log L)^{O(1)}$. All these can be assumed $\leq L^{\theta^{-1}}$ due to the decay $\langle Tq_{\mathfrak{r}}-\zeta\rangle^{-5}$ and the bounds on $\tau$ and $\tau'$. This is equivalent to fixing some real numbers $\sigma_{\mathfrak{n}}=O(L^{\theta^{-1}})$ and requiring the assignment $(k_{\mathfrak{n}}:\mathfrak{n}\in\Tm)$ to satisfy that $|\Omega_{\mathfrak{n}}-\sigma_{\mathfrak{n}}|\leq T^{-1}$ for each $\mathfrak{n}\in\Nm$. Let this final operator, obtained by all the previous reductions, be $\mathcal{G}$. Schematically, the operator $\mathcal{G}$ can be viewed as ``attaching two trees'' $\Tm_1$ and $\Tm_2$ to a single node $\mathfrak{r}'$.
\medskip

{\underline{Step 3: The high order $\mathcal{G}\mathcal{G}^*$ argument.}} For this, we consider the adjoint operator $\mathcal{G}^*$. Similar argument gives a formula for $\mathcal{G}^*$, which is associated with a tree $\mathcal{T}^*$ formed by attaching the two trees $\Tm_2$ and $\Tm_1$ (with $\Tm_2$ on the left of $\Tm_1$) to a single node $\mathfrak{r}'$, in the same way that $\mathcal{G}$ is associated with $\Tm$. Given a large positive integer $D$, we will be considering $(\mathcal{G}\mathcal{G}^*)^{D}$, which is associated with a tree $\Tm^D$. The precise description is as follows.

 \begin{figure}[h!]
  \includegraphics[width=0.58\linewidth]{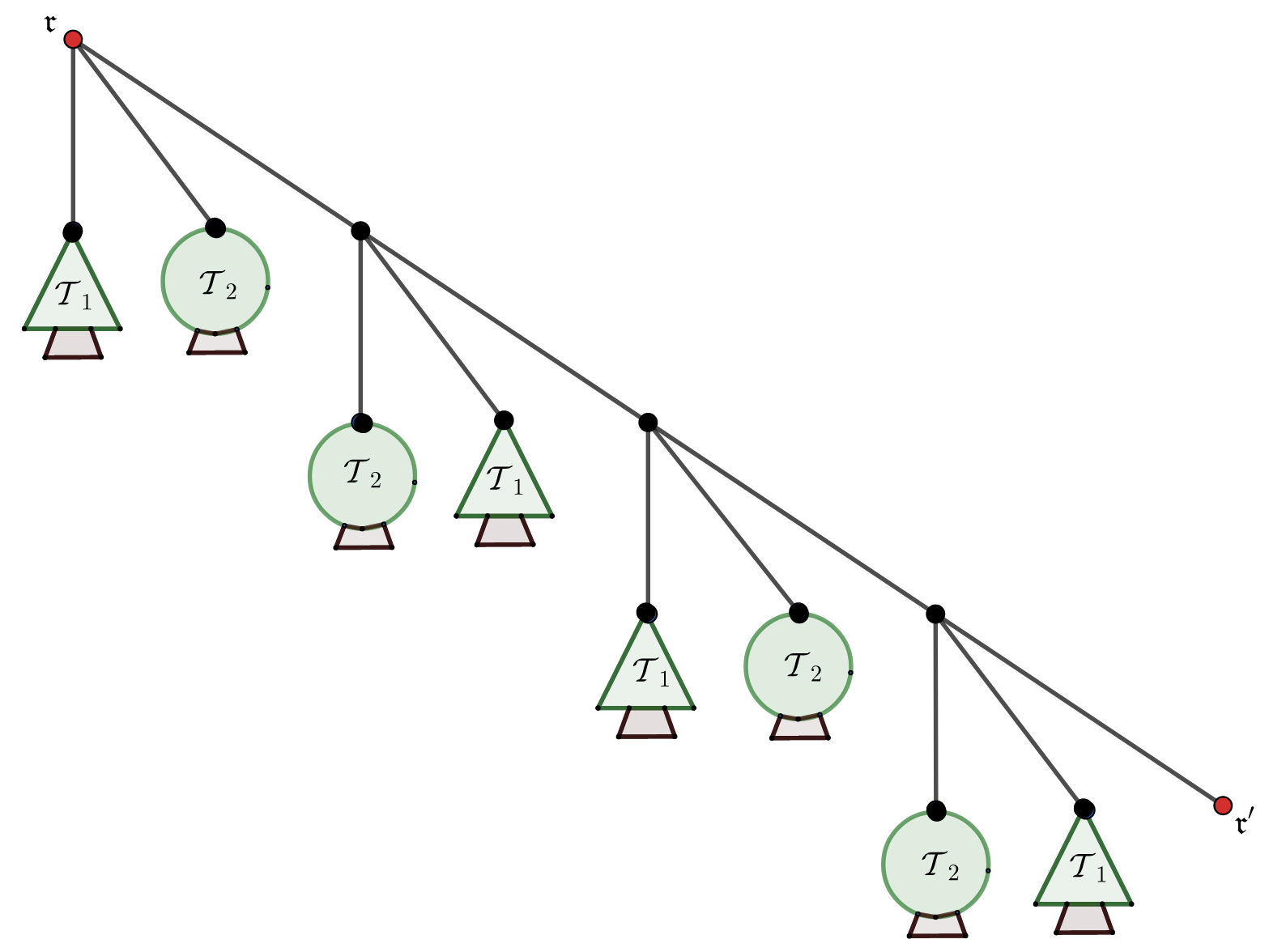}
  \caption{Construction of the tree $\mathcal T^D$ by successive plantings of trees $\mathcal T_1$ and $\mathcal T_2$ onto the first two nodes of a ternary tree starting with a root $\mathfrak r$ and stopping after $2D$ steps leaving a leaf node $\mathfrak r'$. In the figure, $D=2$.}
  \label{fig:manytrees}
\end{figure}

First, $\Tm^D$ is a tree with root node $\mathfrak{r}_0=\mathfrak{r}$, and its first two subtrees (from the left) are $\Tm_1$ and $\Tm_2$. The third subtree has root $\mathfrak{r}_1$, and its first two subtrees (from the left) are $\Tm_2$ and $\Tm_1$. The third subtree has root $\mathfrak{r}_2$, and its first two subtrees (from the left) are $\Tm_1$ and $\Tm_2$, and so on. This process repeats and eventually stops at $\mathfrak{r}_{2D}=\mathfrak{r}'$, which is a single node, and finishes the construction of $\Tm^D$. As usual, denote by $\Lm^D$ and $\mathcal N^D$ the set of leaves and branching nodes respectively. Then the kernel of $(\mathcal{G}\mathcal{G}^*)^{D}$ is given by
\begin{multline}
((\mathcal{G}\mathcal{G}^*)^{D})_{kk'}=\sum_{(k_{\mathfrak{n}}:\mathfrak{n}\in\Tm^D)}\mathcal{K}^{(D)}(|k_{\mathfrak{r}_j}|_\beta^2-|k_{\mathfrak{r}_{j+1}}|_\beta^2:0\leq j\leq 2D-1,\,k_{\mathfrak{l}}:\mathfrak{r}'\neq\mathfrak{l}\in\Lm^D)\\\times\bigg(\frac{\alpha T}{L^{d}}\bigg)^{2D(n_1+n_2+1)}\prod_{\mathfrak{r}'\neq\mathfrak{l}\in\Lm^D}\eta_{k_{\mathfrak{l}}}^{\iota_{\mathfrak{l}}},\end{multline}
where $|\mathcal{K}^{(D)}|\leq L^\theta$ and $|\partial \mathcal{K}^{(D)}|\leq L^\theta T$, and the sum is taken over all admissible assignments $(k_{\mathfrak{n}}:\mathfrak{n}\in\Tm^D)$ that satisfies $(k_{\mathfrak{r}},k_{\mathfrak{r}'})=(k,k')$, that $|k_{\mathfrak{l}}|\leq L^\theta$ for $\mathfrak{r}'\neq\mathfrak{l}\in\Lm$, and that $|\Omega_{\mathfrak{n}}-\sigma_{\mathfrak{n}}|\leq T^{-1}$ for $\mathfrak{n}\in\Nm^D$, where $\sigma_{\mathfrak{n}}=O(L^{\theta^{-1}})$ are fixed. Moreover, $\mathcal{K}^{(D)}$ depends on the variables $k_{\mathfrak{r}}=k$ and $k_{\mathfrak{r}'}=k'$ \emph{only through the quantities} $|k_{\mathfrak{r}_j}|_\beta^2-|k_{\mathfrak{r}_{j+1}}|_\beta^2$ for $0\leq j\leq 2D-1$.

Now, note that each $((\mathcal{G}\mathcal{G}^*)^{D})_{kk'}$ is an explicit multilinear Gaussian expression. Since for fixed $k$ (or $k'$) the number of choices for $k'$ (or $k$) is $O(L^{d+\theta})$, by Schur's estimate we know
\[\|(\mathcal{G}\mathcal{G}^*)^{D}\|_{\ell^2\to\ell^2}\lesssim L^{d+\theta}\sup_{k,k'}|((\mathcal{G}\mathcal{G}^*)^{D})_{kk'}|.\] So it suffices to $L$-certainly bound $|((\mathcal{G}\mathcal{G}^*)^{D})_{kk'}|$ \emph{uniformly} in $k$ and $k'$. We first consider this estimate with fixed $(k,k')$. Applying Lemma \ref{largedev}, we can fix some pairings of $\Tm^D$ and the set $\mathcal{S}^D$ of single leaves, and argue as in the proof of Proposition \ref{lwp1} to conclude $L$-certainly that
\[|((\mathcal{G}\mathcal{G}^*)^{D})_{kk'}|^2\lesssim L^\theta\bigg(\frac{\alpha T}{L^{d}}\bigg)^{4D(n_1+n_2+1)}\sum_{(k_{\mathfrak{l}}:\mathfrak{l}\in\mathcal{S}^D)}\sum_{(k_{\mathfrak{l}}:\mathfrak{l}\in\Lm^D\backslash\mathcal{S}^D)}^{**}\sum_{(k_{\mathfrak{l}'}:\mathfrak{l}\in\Lm^D\backslash\mathcal{S}^D)}^{**'}1,\] 

where the condition for summation, as in the proof of Proposition \ref{lwp1}, is that the unique admissible assignment $(k_{\mathfrak{n}}:\mathfrak{n}\in\mathcal{T}^D)$ determined by $(k_{\mathfrak{l}}:\mathfrak{l}\in\mathcal{S}^D)$ and $(k_{\mathfrak{l}}:\mathfrak{l}\in\Lm^D\backslash\mathcal{S}^D)$ satisfies all the conditions listed above, and the same happens for $(k_{\mathfrak{n}}':\mathfrak{n}\in\mathcal{T}^D)$ corresponding to $(k_{\mathfrak{l}}:\mathfrak{l}\in\mathcal{S}^D)$ and $(k_{\mathfrak{l}}':\mathfrak{l}\in\Lm^D\backslash\mathcal{S}^D)$. We know that $\mathcal T^D$ is a tree of scale $2D(n_1+n_2+1)$ and so $|\Lm^D|=4D(n_1+n_2+1)+1$; let the number of pairings be $p$, then $|\mathcal{S}^D|=4D(n_1+n_2+1)-2p$. By Proposition \ref{countingbd} we can bound\footnote{Strictly speaking we need to modify Proposition \ref{countingbd} a little, as we do not assume $|k'|\leq L^\theta$. But this will not affect the proof, which relies on the translation-invariant Lemma \ref{counting}.} the number of choices for $(k_{\mathfrak{l}}:\mathfrak{l}\in\mathcal{S}^D)$ and $(k_{\mathfrak{l}}':\mathfrak{l}\in\Lm^D\backslash\mathcal{S}^D)$ by $M=L^\theta Q^{4D(n_1+n_2+1)-p}$, and bound the number of choices for $(k_{\mathfrak{l}}':\mathfrak{l}\in\Lm^D\backslash\mathcal{S}^D)$ given $(k_{\mathfrak{l}}:\mathfrak{l}\in\mathcal{S}^D)$ by $M'=L^\theta Q^{p}$. In the end we get, for any \emph{fixed} $(k,k')$, that $L$-certainly 
\[L^{d+\theta}\sup_{k,k'}|((\mathcal{G}\mathcal{G}^*)^{D})_{kk'}|\leq L^{d+\theta}\bigg(\frac{\alpha T}{L^{d}}\bigg)^{2D(n_1+n_2+1)}\left(MM'\right)^{1/2}\leq L^{d+\theta}\rho^{2D(n_1+n_2+1)}.\]

Finally we need to $L$-certainly make the above bound uniform in all choices of $(k,k')$. This is not obvious since we impose no upper bound on $|k|$ and $|k'|$, so the number of exceptional sets we remove in the $L$-certain condition could presumably be infinite. However, note that the coefficient $\boldsymbol{\mathcal{K}}$ depends on $k$ and $k'$ only through the quantities $|k|_\beta^2-|k_{\mathfrak{r}_{j}}|_\beta^2$. Let $\mathcal{D}=\Lm\backslash\{\mathfrak{r'}\}$, then $|k_{\mathfrak{l}}|\leq L^\theta$ for $\mathfrak{l}\in\mathcal{D}$, and the condition for summation restricts that $\big||k|_\beta^2-|k_{\mathfrak{r}_{j}}|_\beta^2\big|\leq L^{\theta^{-1}}$. The reduction from infinitely many possibilities for $k$ (and hence $k'$) to finitely many is done by invoking the following result, whose proof will be left to the end:
\begin{claim}\label{extra0} Let $k\in\mathbb{Z}_L^d$, consider the function \[f_{(k)}:m\mapsto |k|_\beta^2-|k+m|_{\beta}^2,\quad \mathrm{Dom}(f_{(k)})=\big\{m\in\mathbb{Z}_L^d:|m|\leq L^\theta,\big||k|_\beta^2-|k+m|_{\beta}^2\big|\leq L^{\theta^{-1}}\big\},\] then there exists finitely many functions $f_1,\cdots,f_A$ where $A\leq L^{C\theta^{-1}}$, such that for any $k\in\mathbb{Z}_L^d$ there exists $1\leq j\leq A$ such that $|f_{(k)}-f_j|\leq L^{-\theta^{-1}}$ on $\mathrm{Dom}(f_{(k)})$.
\end{claim}Now it is not hard to see that Claim \ref{extra0} allows us to obtain a bound of the form proved above that is uniform in $(k,k')$, after removing at most $O(L^{C\theta^{-1}})$ exceptional sets, each of which having probability $\lesssim e^{-L^\theta}$. This then implies
\[\|(\mathcal{G}\mathcal{G}^*)^{D}\|_{\ell^2\to\ell^2}\lesssim L^{d+\theta}\rho^{2D(n_1+n_2+1)},\] hence
\[\|\mathcal{G}\|_{\ell^2\to\ell^2}\lesssim L^{\frac{d+\theta}{2D}}\rho^{n_1+n_2+1}.\] By fixing $D$ to be a sufficiently large positive integer, we deduce the correct operator bound for $\mathcal{G}$, and hence for $\mathcal{X}$ and $\mathcal{P_+}$. This completes the proof of Proposition \ref{lwp3}.\end{proof}
\begin{proof}[Proof of Claim \ref{extra0}] We will prove the result for any linear function $g(m)=x\cdot m+X$, where $x\in\mathbb{R}^d$ and $X\in\mathbb{R}$ arbitrary. We may also assume $m\in\mathbb{Z}^d$ instead of $\mathbb{Z}_L^d$; the domain $\mathrm{Dom}(g)$ will then be the set $E$ of $m$ such that $|m|\leq L^{1+\theta}$ and $|g(m)|\leq L^{2+\theta^{-1}}$.

Let the affine dimension $\dim(E)=r\leq d$, then $E$ contains a maximal affine independent set $\{q_j:0\leq j\leq r\}$. The number of choices for these $q_j$ is at most $L^{d+1}$, so we may fix them. Let $\Lm$ be the primitive lattice generated by $\{q_j-q_0:1\leq j\leq r\}$, and fix a reduced basis $\{\ell_j:1\leq j\leq r\}$ of $\Lm$. For any $m\in E$ there is a unique integer vector $k=(k_1,\cdots,k_r)\in\mathbb{Z}^r$ such that $|k|\lesssim L^{1+\theta}$, $m-q_0=k_1\ell_1+\cdots+k_r\ell_r$, and as a linear function we can write $g(m)=y\cdot k+Y$, where $y\in\mathbb{R}^r$ and $Y=g(q_0)\in\mathbb{R}$.

Now, let the $k\in\mathbb{Z}^r$ corresponding to $m=q_j$ be $k^{(j)}$, where $1\leq j\leq r$, then since $q_0\in E$ and $q_j\in E$ we conclude that $|y\cdot k^{(j)}|\leq L^{3+\theta^{-1}}$. As the $k^{(j)}$ are linear independent integer vectors in $\mathbb{Z}^r$ with norm bounded by $L^{1+\theta}$, we conclude that $|y|\leq L^{C+\theta^{-1}}$, and consequently $|Y|\leq L^{C+\theta^{-1}}$. We may then approximate $g(m)$ for $m\in E$ by $y_j\cdot k+Y_j$, where $y_j$ and $Y_j$ are one of the $L^{C\theta^{-1}}$ choices that approximate $y$ and $Y$ up to error $L^{-\theta^{-1}}$, and choose $g_{(j)}=y_j\cdot k+Y_j$.
\end{proof}
\subsection{The worst terms}\label{badterm}In this section we exhibit terms $\Jm_\Tm$ that satisfy the lower bound (\ref{lowerbd}). These are the terms corresponding to trees $\Tm$ and pairings (see Remark \ref{decomp1}) as shown in Figure \ref{fig:counterexample}, where $\Tm$ is formed from a single node by successively attaching two leaf nodes, and the `left' node attached at each step is paired with the `right' node attached in the next step. Let the scale $\mathfrak s(\Tm)=r$, then $\Tm$ has exactly $r-1$ pairings. For simplicity we will consider the rational case $\beta_j=1$ and $T\leq L^{2-\delta}$; the irrational case is similar.

  \begin{figure}[h!]
  \includegraphics[width=0.58\linewidth]{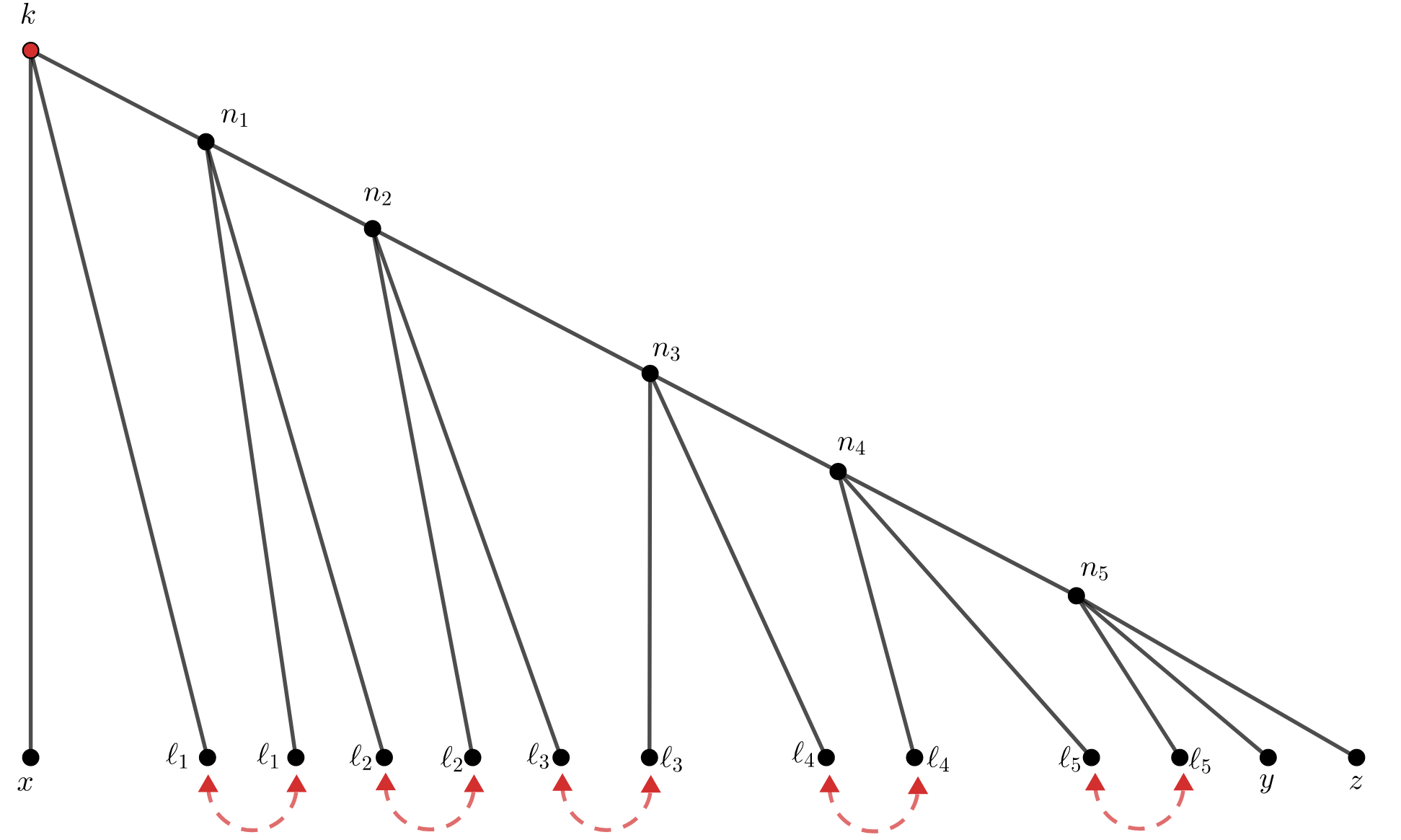}
  \caption{This is a tree of scale $\mathfrak s (\mathcal T)=6$ and $p=6-1=5$ pairings. The pairings force $|y-z|=|n_5-\ell_5|=|n_4-\ell_4|=\ldots=|k-x|$. }
  \label{fig:counterexample}
\end{figure}

Here it is more convenient to work with the time variable $t$ (instead of its Fourier dual $\tau$). To show (\ref{lowerbd}), since $b>1/2$, we just need to bound $(\Jm_\Tm)_k(t)$ from below for some $k$ and some $t\in[0,1]$; moreover since $\chi\equiv 1$ on $[0,1]$, and using the recursive definition (\ref{defjt}), we can write
\begin{equation}\label{formulabadjt1}(\Jm_\Tm)_k(t)=\bigg(\frac{\alpha T}{L^d}\bigg)^r\sum_{x-y+z=k}\bigg(\sum_{\ell_1,\cdots,\ell_{r-1}}\mathcal{A}\cdot\prod_{j=1}^{r-1}n_{\mathrm{in}}(\ell_j)|\eta_{\ell_j}|^2\bigg)\cdot \eta_x\overline{\eta_y}\eta_z \cdot \sqrt{n_{\mathrm{in}}(x)n_{\mathrm{in}}(y)n_{\mathrm{in}}(z)},
\end{equation}
where the variables in the summation satisfy (due to admissibility) that
\[k-x=n_1-\ell_1=n_2-\ell_2=\cdots=n_{r-1}-\ell_{r-1}=y-z:=q,\] and the coefficient $\mathcal{A}$ is given by
\begin{equation}\label{formulabadjt2}\mathcal{A}=\int_{t>t_1>\cdots >t_r>0}e^{2\pi iT(t_1\Omega_1+\cdots +t_r\Omega_r)}\,\mathrm{d}t_1\cdots\mathrm{d}t_r,
\end{equation} with $\Omega_j$ being the resonance factors, namely
\[\Omega_1=2q\cdot(k-n_1),\,\Omega_2=2q\cdot(n_1-n_2),\cdots,\,\Omega_{r-1}=2q\cdot(n_{r-2}-n_{r-1}),\,\Omega_r=2q\cdot(n_{r-1}-z).\]

In (\ref{formulabadjt1}) we may replace $|\eta_{\ell_j}|^2$ by 1, so the factor in the big parenthesis, denoted by $\mathcal{A}_{kxyz}$, involves no randomness. Therefore with a high probability,
\[|(\Jm_\Tm)_k(t)|^2\sim\bigg(\frac{\alpha T}{L^d}\bigg)^{2r}\sum_{x-y+z=k}|\mathcal{A}_{kxyz}|^2.\] In the above sum we may fix $q\in \mathbb{Z}_L^d$ with $0<|q|\lesssim L^{-1}$ which has $O(1)$ choices, and write
\begin{multline*}\mathcal{A}_{kxyz}=\int_{t>t_1>\cdots >t_r>0}e^{4\pi iTq\cdot[t_1(k-q)+t_r(q-z)]}\,\mathrm{d}t_1\cdots\mathrm{d}t_r\\\times\sum_{\ell_1,\cdots,\ell_{r-1}}e^{4\pi iT [(t_2-t_1)q\cdot \ell_1+\cdots +(t_r-t_{r-1})q\cdot\ell_{r-1}]}\prod_{j=1}^{r-1}n_{\mathrm{in}}(\ell_j).\end{multline*} By Poisson summation, and noticing that $|Tt_jq|\lesssim L^{1-\delta}$, we conclude that up to constants
\[\sum_{\ell_1,\cdots,\ell_{r-1}}e^{4\pi iT [(t_2-t_1)q\cdot \ell_1+\cdots +(t_r-t_{r-1})q\cdot\ell_{r-1}]}\prod_{j=1}^{r-1}n_{\mathrm{in}}(\ell_j)=L^{(r-1)d}\prod_{j=1}^{r-1}\widehat{n_{\mathrm{in}}}(Tq(t_{j+1}-t_j))+O(L^{-\infty}).\] By making change of variables $s_j=t_{j}-t_{j+1}\,(1\leq j\leq r-1)$ and $s_0=t-t_1$, $s_r=t_r$, we can reduce
\[\mathcal{A}_{xyz}\approx L^{(r-1)d}\int_{s_0+\cdots +s_{r}=t} e^{4\pi iT[(t-s_0)q\cdot(k-q)+s_rq\cdot(q-z)]}\prod_{j=1}^{r-1}\widehat{n_{\mathrm{in}}}(Tqs_j).\] By choosing some particular $(k,q,z)$ we may assume $q\cdot(k-q)=q\cdot(q-z)=0$, and if we also choose $n_{\mathrm{in}}$ such that $\widehat{n_{\mathrm{in}}}$ is positive, say $n_{\mathrm{in}}(k)=e^{-|k|^2}$, and $t=\min(1,LT^{-1})$, then we have \[|\mathcal{A}_{kxyz}|\sim L^{(r-1)d}\min(1,LT^{-1})^{r},\] and hence with high probability
\[\sup_{k,t}|(\Jm_\Tm)_k(t)|\gtrsim L^{-d}\min(\alpha T,\alpha L)^r=L^{-d}\rho^r\] for any fixed $r$, hence (\ref{lowerbd}).
\begin{rem}\label{decomp1}Here, strictly speaking, we are further decomposing $\Jm_\Tm$ into the sum of terms $\Jm_{\Tm,\mathcal{P}}$ where $\mathcal{P}$ represents the pairing structure of $\Tm$. In the proof of Proposition \ref{lwp1}, we are actually making the same decomposition (by identifying the set of pairings) and proving the same bound for each $\Jm_{\Tm,\mathcal{P}}$. On the other hand the example here shows that individual terms $\Jm_{\Tm,\mathcal{P}}$ can be very large in absolute value. Thus to get any improvement to the results of this paper, one would need to explore the subtle cancellations between the $\Jm_{\Tm,\mathcal{P}}$ terms with different $\Tm$ or different $\mathcal{P}$.
\end{rem}

\bigskip

\section{Proof of the main theorem}\label{proof of theorem} 

\bigskip 

In this section we prove Theorem \ref{thm2} (which also implies Theorem \ref{thm1}). Since we may alter the value of $T$, in proving Theorem \ref{thm2} we may restrict to the case $T/2\leq t\leq T$.

\medskip

First note that $\mathbb E|\widehat u(k, t)|^2=\mathbb E |a_k(s)|^2$, where $s:=\frac{t}{T} \in [1/2,1]$. By mass conservation, we have that $L^{-d/2}\sum_{k \in \Z^d_L}|a_k|^2=O(1)$ and hence $\|a_k\|_{\ell^\infty}\lesssim L^{d/2}$. As such, if we denote by $\Gamma$ the intersection of all the $L$-certain events in Propositions \ref{lwp0} and \ref{lwp1}, we have for $0\leq s \leq 1$ (denoting by $\mathbb E_\Gamma G= \mathbb E \mathbf 1_\Gamma G$)
\begin{equation}\label{expansion}
\begin{aligned}
\mathbb E|\widehat u(k, Ts)|^2&=\mathbb{E}_\Gamma[|(\Jm_0)_k(s)|^2+|(\Jm_1)_k(s)|^2+2\mathrm{Re}\overline{(\Jm_0)_k(s)}(\Jm_1)_k(s)+2\mathrm{Re}\overline{(\Jm_0)_k(s)}(\Jm_2)_k(s)]\\
&+\sum_{3\leq n\leq N}2\mathbb{E}\mathrm{Re}\overline{(\Jm_0)_k(s)}(\Jm_n)_k(s)+\sum_{1\leq n_1,n_2\leq N;n_1+n_2\geq 3}\mathbb{E}_\Gamma\overline{(\Jm_{n_1})_k(s)}(\Jm_{n_2})_k(s)\\
&+\sum_{n\leq N}2\mathbb{E}_\Gamma\mathrm{Re}\overline{(\Rm_{N+1})_k(s)}(\Jm_n)_k(s)+\mathbb{E}_\Gamma|(\Rm_{N+1})_k(s)|^2+O(e^{-L^\theta}).
\end{aligned}
\end{equation} By using Proposition \ref{lwp0} we can bound the last three terms by
\[\big|\mathbb{E}_\Gamma\overline{(\Jm_{n_1})_k(s)}(\Jm_{n_2})_k(s)\big|\leq L^{\theta +c(b-1/2)}\rho^{n_1+n_2-2}(\alpha^2T)\leq L^{-\delta}/10\frac{T}{T_{\mathrm{kin}}},\] 
\[\big|\mathbb{E}_\Gamma\mathrm{Re}\overline{(\Rm_{N+1})_k(s)}(\Jm_n)_k(s)\big|+\mathbb{E}_\Gamma|(\Rm_{N+1})_k(s)|^2\leq L^{\theta+C(b-1/2)}\rho^N\leq L^{-10d}.\] As with the first term on the second line of (\ref{expansion}), since $(\Jm_0)_k(s)=\chi(t)\sqrt{n_{\mathrm{in}}}\cdot\eta_k(\omega)$, by direct calculations and similar arguments as in the proof of Proposition \ref{lwp1} we can bound, for any tree $\Tm$ with $\mathfrak s (\Tm)=n$, that \[\big|\mathbb{E}\overline{(\Jm_0)_k(s)}(\Jm_{\Tm})_k(s)\big|\leq L^\theta\bigg(\frac{\alpha T}{L^{d}}\bigg)^nM,\] where $M$ is the quantity estimated in Proposition \ref{countingbd} (i.e. the number of strongly admissible assignments satisfying (\ref{admissibility2})), with all but one leaf of $\Tm$ being paired, and $\mathcal{R}=\{\mathfrak{r}\}$. By Corollary \ref{imprcor} we have
\[\big|\mathbb{E}\overline{(\Jm_0)_k(s)}(\Jm_{\Tm})_k(s)\big|\leq  L^\theta\bigg(\frac{\alpha T}{L^{d}}\bigg)^nQ^{n-2}L^{2d}T^{-1}\leq L^{\theta}\rho^{n-2}(\alpha^2T)\leq L^{-\delta}/10\frac{T}{T_{\mathrm{kin}}}.\]

It then suffices to calculate the main term, which is the first line of (\ref{expansion}). Up to an error of size $O(e^{-L^\theta})$ we can replace $\mathbb{E}_\Gamma$ by $\mathbb{E}$; also we can easily show that $\mathbb{E}\overline{(\Jm_0)_k(s)}(\Jm_1)_k(s)=0$. For $|s|\leq 1$ clearly $\mathbb{E}|(\Jm_0)_k(s)|^2=n_{\mathrm{in}}$; as for the other two terms, namely $\mathbb{E}|(\Jm_1)_k(s)|^2$ and $2\mathbb{E}\mathrm{Re}\overline{(\Jm_0)_k(s)}(\Jm_2)_k(s)$, we compute as follows: Recall that $(a_{\mathrm{in}})_{k}=\sqrt{n_{\mathrm{in}}(k)}\eta_k(\omega)$ and
\begin{align*}
(\Jm_1)_k(s)=-\frac{\alpha T}{L^{d}}\biggl[\sum_{(k_1, k_2, k_3); \Omega\neq 0}^\times (a_{\mathrm{in}})_{k_1}\overline{(a_{\mathrm{in}})_{k_2}}(a_{\mathrm{in}})_{k_3} \frac{e^{2\pi iT\Omega s}-1}{2\pi T\Omega}+is \sum_{(k_1, k_2, k_3); \Omega= 0}^\times (a_{\mathrm{in}})_{k_1}\overline{(a_{\mathrm{in}})_{k_2}}(a_{\mathrm{in}})_{k_3}\\
-is |(a_{\mathrm{in}})_{k}|^2(a_{\mathrm{in}})_{k} \biggr]
\end{align*}
and as such, we have that
\begin{align*}
\mathbb{E}|(\Jm_1)_k(t)|^2=&\frac{\alpha^2s^2T^2}{L^{2d}}\biggl[\sum_{(k_1, k_2, k_3); \Omega\neq 0}^\times (n_{\mathrm{in}})_{k_1}\overline{(n_{\mathrm{in}})_{k_2}}(n_{\mathrm{in}})_{k_3} \left|\frac{\sin \pi \Omega Ts}{\pi \Omega Ts}\right|^2+ \sum_{(k_1, k_2, k_3); \Omega= 0}^\times (n_{\mathrm{in}})_{k_1}\overline{(n_{\mathrm{in}})_{k_2}}(n_{\mathrm{in}})_{k_3}\\
&\qquad \qquad \qquad \qquad \qquad \qquad \qquad \qquad  \qquad \qquad \qquad \qquad \qquad \qquad + |(n_{\mathrm{in}})_{k}|^2(n_{\mathrm{in}})_{k} \biggr]\\
=&\frac{\alpha^2t^2}{L^{2d}}\sum_{(k_1, k_2, k_3); \Omega\neq 0}^\times (n_{\mathrm{in}})_{k_1}\overline{(n_{\mathrm{in}})_{k_2}}(n_{\mathrm{in}})_{k_3} \left|\frac{\sin \pi \Omega t}{\pi \Omega t}\right|^2 +O(\frac{T}{T_{\mathrm kin}}L^{-\delta}),
\end{align*}
where we used that $T<L^{2d-\delta}$ for the third term, and estimated the second term by $L^{2d-2+\theta}$ for general $\beta_j$ and by $L^{d+\theta}$ if $\beta_j$ are irrational (use for example Lemma \ref{counting} with $m=0$ and $T=L^2$ and $L^d$ respectively). 

A similar computation for $2\mathbb{E}\mathrm{Re}\overline{(\Jm_0)_k(s)}(\Jm_2)_k(s)$ (see Section 3 of \cite{BGHS2}) gives

\[\mathbb{E}[|(\Jm_1)_k(t)|^2+2\mathrm{Re}\overline{(\Jm_0)_k(t)}(\Jm_2)_k(t)]=\frac{\alpha^2t^2}{L^{2d}}\cdot\mathscr{S}_{t}(n_{\mathrm{in}}) +O(\frac{T}{T_{kin}}L^{-\delta}),\] where \begin{equation}\label{Riemannsum}
 \mathscr S_t(\phi):=\sum\limits_{\substack{k_i\in\mathbb{Z}_L^d\\ k-k_1+k_2-k_3=0}}\phi_k \phi_{k_1} \phi_{k_2} \phi_{k_3}  \left[ \frac{1}{\phi_k} - \frac{1}{\phi_{k_1}} + \frac{1}{\phi_{k_2}} - \frac{1}{\phi_{k_3}} \right]
\left| \frac{\sin(\pi t\Omega(\vec k)}{\pi t \Omega(\vec k)} \right|^2,
\end{equation}
with $\Omega(\vec k)=\Omega(k,k_1,k_2,k_3)=|k_1|_\beta^2-|k_2|_\beta^2+|k_3|_\beta^2-|k|_\beta^2$. Therefore we conclude that
\[\mathbb{E}|\widehat{u}(k,t)|^2=n_{\mathrm{in}}+\frac{\alpha^2t^2}{L^{2d}}\mathscr{S}_{t}(n_{\mathrm{in}})+O(\frac{T}{T_{\mathrm kin}}L^{-\delta/10}).\]

In the following section, we will derive the asymptotic for the sum $\mathscr S_t$, namely we will show that $\mathscr S_t(\phi)=\mathscr K_t(\phi)+O(t^{-1}L^{2d-\theta})$ for some $\theta>0$ where $\mathscr K_t$ is given by 

\begin{equation}\label{Kintegral}
\mathscr K_t(\phi):=L^{2d} \int_{\xi_1-\xi_2+\xi_3=\xi }\phi(\xi) \phi(\xi_1) \phi(\xi_2) \phi(\xi_3)  \left[ \frac{1}{\phi(\xi)} - \frac{1}{\phi(\xi_1)} + \frac{1}{\phi(\xi_2)} - \frac{1}{\phi(\xi_3)} \right]    \left| \frac{\sin(\pi t\Omega(\vec \xi))}{\pi t\Omega(\vec \xi)} \right|^2     d\xi_1 d\xi_2 \,d\xi_3.
\end{equation}

Finally, the proof is complete by using the fact that for a smooth function $f$, 
\[
t \int \left| \frac{\sin(\pi t x)}{\pi t x} \right|^2 f(x)\,dx = \pi^2 f(0) + O(t^{-1}).\qedhere
\]

\bigskip

\section{Number Theoretic Results}\label{number theory section} 

\bigskip

The purpose of this section is to prove the asymptotic formula for $\mathscr{S}_t$ defined in \eqref{Riemannsum}. The sum $\mathscr{S}_t$ should be regarded as a Riemann sum that approximates the integral  $\mathscr K_t$ in \eqref{Kintegral}. However, this approximation is far from trivial because of the highly oscillating factor $\left| \frac{\sin(\pi t\Omega(\vec \xi))}{\pi t\Omega(\vec \xi)} \right|^2$, which makes the problem intimately related to the equidistribution properties of the values quadratic form $\Omega$.

 \medskip

\begin{thm}\label{asymptotic}
Let $\phi\in \mathcal S(\R^d)$ with $d\geq 3$. For any $\delta>0$, there exists $\theta>0$ such that the asymptotic holds: 
\begin{enumerate}
\item (general tori) For any $\beta_i\in [1, 2]^d$, and any $t< L^{2-\delta}$, there holds that
$$
\mathscr S_t=\mathscr K_t +O(L^{2d-\theta} t^{-1}).
$$

\item (generic tori) For generic  $\beta_i\in [1, 2]^d$, and any $t< L^{d-\delta}$, there holds
$$
\mathscr S_t=\mathscr K_t +O(L^{2d-\theta} t^{-1}).
$$
\end{enumerate}
\end{thm}

 It is not hard to see that $\mathscr K_t=O(L^{2d})t^{-1}$, which justifies the sufficiency of the error term bound above. 
 
 \begin{rem}
 It is interesting that, in the case of the rational torus for which $\beta_j=1$, the above asymptotic ceases to be true at the end point $t=L^2$. This corresponds to $\mu=1$ in \eqref{Rsum} below, whose asymptotic was studied in \cite{FGH, BGHS1} and yields a logarithmic divergence when $d=2$ and a different multiplicative constant for $d\geq 3$ compared to the asymptotic in the above theorem. 
 \end{rem}

\begin{proof}

The proof of part (2) is contained in Theorem 8.1 of \cite{BGHS2}. As such, we will only focus on the first part, which is less sophisticated. To simplify the notation, we will drop the subscript $t$ from $\mathscr S_t$ and $\mathscr K_t$. We use a refinement of Lemma 8.10 in \cite{BGHS2} which basically covers the case $t<L^{1-\delta}$. First, one observes that $\Omega(\vec k)=-2\mathcal Q (k_1-k,k_3-k)$ where $\mathcal Q(x,y):=\sum_{j=1}^d \beta_j x_j y_j$. As such, changing variables $N_1=L(k_1-k)\in \Z^d$ and $N_2=L(k_3-k)\in \Z^d$, we will write the sum $\mathscr S$ in the form 
\begin{equation}\label{Rsum}
\mathscr S=\sum_{N=(N_1, N_2)\in \Z^{2d}}W(\frac{N}{L}) g\left(4\mu \mathcal Q(N)\right); \quad g(x):=\left|\frac{\sin \pi x}{\pi x}\right|^2, \quad \mu:=\frac12 tL^{-2}<L^{-\delta}.
\end{equation}
where $W\in \mathcal S(\R^{2d})$. As such, we have 
\begin{equation}\label{Kint}
\mathscr K=L^{2d}\int_{(z_1, z_2) \in \R^{2d}} W(z_1, z_2) g(4\mu z_1 \cdot z_2)\; dz_1 dz_2. 
\end{equation}

$\bullet$ \emph{Step 1 (Truncation in $N$)}: We first notice that the main contribution of the sum $\mathscr S$ (resp. the integral $\mathscr K$) comes from the region $|N|\lesssim L^{1+\delta_1}$ (resp. $|(\xi_1, \xi_2)|\lesssim L^{\delta_1}$) where $\delta_1=\frac{\delta}{100}$. This uses the fact that $W$ is a Schwartz function with sufficient decay. As such, we can include (without loss of generality) in the sum $\mathscr S$ (resp. the integral $\mathcal K$) a factor $\chi(\frac{N}{L^{1+\delta_1}})$ (resp. $\chi(\frac{z}{L^{1+\delta_1}})$), where $\chi\in C_c^\infty(\R^d)$ is 1 on the unit ball $B(0, \frac{1}{10})$ and vanishes outside $B(0, \frac{2}{10})$.

\medskip

$\bullet$ \emph{Step 2 (Isolating the main term)} We now use that the Fourier transform of $g$ is given by the tent function $\widehat g(x)=1-|x|$ on the interval $[-1,1]$ and vanishes otherwise, to write (using the notation $e(x):=e^{2\pi i x}$)
\begin{align*}
\mathscr S=&\sum_{N=(N_1, N_2)\in \Z^{2d}}W(\frac{N}{L}) \chi(\frac{N}{L^{1+\delta_1}}) \int_{-1}^1 \widehat g(\tau) e\left(4\mu \tau \mathcal Q(N)\right) d\tau\\
=&\mu^{-1}\sum_{N=(N_1, N_2)\in \Z^{2d}}W(\frac{N}{L}) \chi(\frac{N}{L^{1+\delta_1}}) \int_{-\mu}^\mu \widehat g(\frac{\tau}{\mu}) e\left(4 \tau \mathcal Q(N)\right) d\tau\\
=&\mathscr S_A+\mathscr S_B,
\end{align*}
where $A$ is the contribution of $|\tau|\leq L^{-1-\delta_1}$ and $B$ is the contribution of the complementary region, which could be empty if $\mu<L^{-1-\delta_1}$ in which case we assume $B=0$. By Poisson summation, we have that 

\begin{align*}
\mathscr S_A=&\mu^{-1} \int_{|\tau|\leq \min(\mu, L^{-1-\delta_1})}\widehat g(\frac{\tau}{\mu}) \sum_{c\in \Z^{2d}}\int_{z\in \R^{2d}} W(\frac{z}{L})\chi(\frac{z}{L^{1+\delta_1}})e\left(4 \tau \mathcal Q(z)-c\cdot z\right)\; dz \, d\tau\\
=&\mu^{-1} L^{2d}\int_{|\tau|\leq \min(\mu, L^{-1-\delta_1})}\widehat g(\frac{\tau}{\mu}) \int_{z\in \R^{2d}} W(z)\chi(\frac{z}{L^{\delta_1}})e\left(4 \tau L^2 \mathcal Q(z)\right)\; dz \, d\tau\\
&+\mu^{-1} L^{2d} \int_{|\tau|\leq \min(\mu, L^{-1-\delta_1})}\widehat g(\frac{\tau}{\mu}) \sum_{\substack{c\in \Z^{2d}\\c\neq 0}}\int_{z\in \R^{2d}} W(z)\chi(\frac{z}{L^{\delta_1}})e\left(4 \tau L^2 \mathcal Q(z)-Lc\cdot z\right)\; dz \, d\tau\\
=&\mu^{-1} L^{2d}\int_{|\tau|\leq \mu}\widehat g(\frac{\tau}{\mu}) \int_{z\in \R^{2d}} W(z)\chi(\frac{z}{L^{\delta_1}})e\left(4 \tau L^2 \mathcal Q(z)\right)\; dz \, d\tau\\
&+\mu^{-1} L^{2d}\int_{ \min(\mu, L^{-1-\delta_1})<|\tau|\leq \mu}\widehat g(\frac{\tau}{\mu}) \int_{z\in \R^{2d}} W(z)\chi(\frac{z}{L^{\delta_1}})e\left(4 \tau L^2\mathcal Q(z)\right)\; dz \, d\tau\\\
&+\mu^{-1} L^{2d}\int_{|\tau|\leq \min(\mu, L^{-1-\delta_1})}\widehat g(\frac{\tau}{\mu}) \sum_{\substack{c\in \Z^{2d}\\c\neq 0}}\int_{z\in \R^{2d}} W(z)\chi(\frac{z}{L^{\delta_1}})e\left(4 \tau L^2 \mathcal Q(z)-Lc\cdot z\right)\; dz \, d\tau\\
=&\mathscr K +\mathscr S_{A1}+\mathscr S_{A2},
\end{align*}
where $\mathscr S_{A1}$ and $\mathscr S_{A2}$ are respectively the second and third terms in second to last equality. 

The remainder of the proof is to show that $\mathscr S_{A1}, \mathscr S_{A2}$ and $\mathscr S_{B}$ are error terms. 

\medskip

$\bullet$ \emph{Step 3 ($\mathscr S_{A1}$ and $\mathscr S_{A2}$ are error terms)}. To estimate $\mathscr S_{A1}$, we use the stationary phase estimate
 
$$\left|\int_{z\in \R^{2d}} W(z)\chi(\frac{z}{L^{\delta_1}})e\left(4 \tau L^2\mathcal Q(z)\right)\; dz\right| \lesssim (\tau L^2)^{-d}$$
and the fact that the term in only nonzero if $\mu > L^{-1-\delta_1}$, to bound
$$
\left|\mathscr S_{A1}\right|\lesssim \mu^{-1}\int_{  L^{-1-\delta_1} <|\tau|\leq \mu}|\widehat g(\frac{\tau}{\mu})| |\tau|^{-d} \, d\tau\lesssim \mu^{-1}L^{(1+\delta_1)(d-1)}=t^{-1}L^{d+1 +\delta_1(d-1)}\ll L^{2d-\delta}t^{-1}.
$$

For $\mathscr S_{A2}$, we use non-stationary phase techniques relying on the fact that the phase function $\Phi(z)=4\tau L^2 \mathcal Q(z) -Lc\cdot z$ satisfies $|\nabla_z \Phi(z)|=|L(4\tau L (z_2, z_1)-c)|\gtrsim L|c|$ for $c\neq 0$ since $|z|\leq \frac{L^{\delta_1}}{5}$. Therefore, one can integrate by parts in $z$ sufficiently many times and show that $\left|\mathscr S_{A1}\right| \ll L^{2d-\delta}t^{-1}$ as well. 

$\bullet$ \emph{Step 4 ($\mathscr S_{B}$ is an error terms)}. Here, we assume w.l.o.g. that $L^{-1-\delta_1} <\mu\leq L^{-\delta}$ (otherwise $\mathscr S_{B}=0$). Therefore,
$$
\mathscr S_{B}=\mu^{-1}\int_{L^{-1-\delta_1}<|\tau|\leq \mu} \widehat g(\frac{\tau}{\mu})F(\tau) d\tau, \qquad F(\tau)=\sum_{N=(N_1, N_2)\in \Z^{2d}}W(\frac{N}{L}) \chi(\frac{N}{L^{1+\delta_1}})   e\left(4 \tau \mathcal Q(N)\right).
$$
Recall that, $\mathcal Q(N)=\sum_{j=1}^d\beta_j (N_1)_j(N_2)_j$, so we will perform the following change of variables 
$$
(N_1)_j=\frac{p_j+q_j}{2}, \qquad (N_2)_j=\frac{p_j-q_j}{2}, \qquad p_j=q_j (\text(mod)2).
$$
As such the sum in $(N_1)_j, (N_2)_j \in \Z^{2}$, will become a sum 
$$
\sum_{(p_j, q_j)\in \Z^{2}} -\sum_{p_j \in 2\Z, q_j \in \Z}-\sum_{p_j \in \Z, q_j \in 2\Z} +2\sum_{(p_j, q_j)\in 2\Z^{2}}.
$$

We will estimate the contribution of the first sum, and it will be obvious from the proof that the other sums are estimated similarly. Also, by symmetry, we only need to consider the sums for which $p_j, q_j \geq 0$, which reduces us to 
$$
F(\tau)=\sum_{\substack{p_j, q_j\geq 0 \\ j=1, \ldots, d}} \widetilde W(\frac{(p,q)}{L}) \widetilde \chi(\frac{(p,q)}{L^{1+\delta_1}})\prod_{j=1}^de(\tau \beta_j p_j^2)e(-\tau \beta_j q_j^2)
$$
Let $G(s, n)=\sum_{p=0}^ne(s p^2)$ be the Gauss sum, and abusing notation, also denote by $G(s, x)=G(s, [x])$ for $x\in \R$ where $[x]$ is the floor function. Then, 
\begin{align*}
F(\tau)=\int_{\substack{x_j,y_j\geq 0\\j=1, \ldots d}}\widetilde W(\frac{(x,y)}{L}) \widetilde \chi(\frac{(x,y)}{L^{1+\delta_1}})\prod_{j=1}^d G'(\tau\beta_j, x_j)G'(\tau\beta_j, y_j)
\end{align*}

Integrating by parts in all the variables (equivalently performing an Abel summation), one obtains that

\begin{align*}
F(\tau)=\int_{\substack{x_j,y_j\geq 0\\j=1, \ldots d}}\partial_{x_1}\ldots \partial_{y_d}\widetilde W(\frac{(x,y)}{L}) \widetilde \chi(\frac{(x,y)}{L^{1+\delta_1}})\prod_{j=1}^d G(\tau\beta_j, x_j)G(\tau\beta_j, y_j) +l.o.t.
\end{align*}
where the l.o.t. are lower order terms that can be bounded is a similar or simpler way than the main term above.

We now recall the Gauss sum estimate for $G(s, n)$: Let $0\leq a<q\leq n$ be integers such that $(a, q)=1$ and $|s-\frac{a}{q}|<\frac{1}{qn}$ (for any $s$ and $n$ such a pair exists by Dirichlet's approximation theorem), then 
$$|G(s, n)|\leq \frac{n}{\sqrt q (1+n\|s-\frac{a}{q}\|^{1/2})}\leq \frac{n}{\sqrt q}.$$

Here $s=\tau \beta_j$ with $\tau\in [L^{-1-\delta_1}, L^{-\delta}]$, $\beta_j \in [1,2]$. This means that either $|n|<L^{2\delta}$ or $\frac{a}{q}\lesssim L^{-\delta}$( $\Rightarrow q\gtrsim L^{\delta}$), and in either case we get that $|G(s, n)|\lesssim L^{(1+\delta_1-\frac{\delta}{2})}$ since $|n|\lesssim L^{1+\delta_1}$ (note that the above argument works when $a>0$; if $a=0$ we have the better bound $|G(s,n)|\lesssim |s|^{-1/2}\leq L^{2/3}$).

\medskip 

As a result, we have 
$$
|F(\tau)|\lesssim L^{(1+\delta_1-\frac{\delta}{2})(2d-4)}\int_{\substack{x_j,y_j\geq 0\\j=1, \ldots d}}\left|\partial_{x_1}\ldots \partial_{y_d}\widetilde W(\frac{(x,y)}{L}) \widetilde \chi(\frac{(x,y)}{L^{1+\delta_1}}) \right| \prod_{j=1}^2 |G(\tau\beta_j, [x_j])||G(\tau\beta_j, [y_j])|
$$

This gives that 
$$
\left|\mathscr S_{B}\right|\lesssim L^{(1+\delta_1-\frac{\delta}{2})(2d-4)}\mu^{-1}  \int_{\substack{x_j,y_j\geq 0\\j=1, \ldots d}}\left|\partial_{x_1}\ldots \partial_{y_d}\widetilde W(\frac{(x,y)}{L}) \widetilde \chi(\frac{(x,y)}{L^{1+\delta_1}}) \right| \int_{|\tau|\leq \mu}\prod_{j=1}^2 |G(\tau\beta_j, [x_j])||G(\tau\beta_j, [y_j])| \, d\tau.
$$

Now using Hua's lemma (cf.~Lemma 20.6 \cite{IwKow}), we have that $\|G(\tau , n_j)\|_{L^4[0, 1]}\lesssim n_j^{1/2+\delta_1}\lesssim L^{(1+\delta_1)(\frac{1}{2}+\delta_1)}$, which gives that 
$$
\left|\mathscr S_{B}\right|\lesssim L^{(1+\delta_1-\frac{\delta}{2})(2d-4)}\mu^{-1}  L^{(1+\delta_1)(2+4\delta_1)}=L^{2d}t^{-1}L^{-\delta(d-2)+\delta_1(2d-2+4\delta_1)}\ll L^{2d-\theta}t^{-1}
$$
provided that $\theta<\min(1, \frac{(d-2)\delta}{2d})$, and recalling that $\delta_1=\frac{\delta}{100}$.

\medskip

\end{proof}

\medskip

\noindent {\bf Acknowledgments.} The authors would like to thank Andrea Nahmod, Sergey Nazarenko, and Jalal Shatah for illuminating discussions. Y. Deng was supported by NSF grant DMS 1900251. Z. Hani was supported by NSF grants DMS-1852749 and DMS-1654692, a Sloan Fellowship, and the ``Simons Collaboration Grant on Wave Turbulence". The results of this work were announced on Nov.~1, 2019 at a Simons Collaboration Grant meeting.

\bigskip

\Addresses

 \end{document}